\documentclass[12pt]{article}
\usepackage{amssymb}
\usepackage{amsfonts}
\usepackage{amsmath}
\usepackage[usenames]{color}
\usepackage{mathrsfs}
\usepackage{amsfonts}
\usepackage{amssymb,amsmath}
\usepackage{CJK}
\usepackage{cite}
\usepackage{cases}
\usepackage{amsthm}

\def\cl{\centerline}

\def\vs{\vspace*}

\def\R{\mathfrak{g}}

\def\Z{\mathbb{Z}}

\def\H{\mathfrak{hv}}

\def\C{\mathbb{C}}

\def\ni{\noindent}

\numberwithin{equation}{section}
\newtheorem{theo}{Theorem}[section]
\newtheorem{defi}[theo]{Definition}
\newtheorem{coro}[theo]{Corollary}
\newtheorem{lemm}[theo]{Lemma}
\newtheorem{prop}[theo]{Proposition}
\newtheorem{clai}{Claim}
\newtheorem{case}{Case}

\newtheorem{rema}[theo]{Remark}

\begin{document}
\begin{center}
\cl{\large\bf \vs{8pt}  Representations of the N=1 Heisenberg-Virasoro superalgebra}
\cl{Ziqi Hong, Haibo Chen$^*$ and Yucai Su}
\end{center}
\footnote{$^{*}$ Corresponding author (H. Chen).}
{\small
\parskip .005 truein
\baselineskip 3pt \lineskip 3pt

\noindent{{\bf Abstract:}
We first  define a class  of  non-weight  modules over   the N=1 Heisenberg-Virasoro superalgebra $\mathfrak{g}$, which are reducible modules.
 Then we give all submodules of such modules, and present the corresponding
irreducible quotient modules which were exactly studied in \cite{DL}. Also, we prove that those modules constitute a
complete classification of     $U(\mathfrak{h})$-free modules of rank $2$ over $\mathfrak{g}$, where $\mathfrak{h}=\C L_{0}\oplus\C H_{0}$ is a degree-0 subalgebra  of  $\mathfrak{g}$. As an application,  we   obtain a class of weight $\mathfrak{g}$-modules      from those non-weight $\mathfrak{g}$-modules by the  weighting functor. Furthermore,  we study the  non-weight modules  over the four  subalgebras of $\mathfrak{g}$: (i) the Heisenberg-Virasoro  algebra; (ii) the Neveu-Schwarz algebra; (iii) the Fermion-Virasoro algebra; (iv) the Heisenberg-Clifford superalgebra. As far as we know, those  non-weight Heisenberg-Virasoro  modules were
 constructed      in  \cite{HCS}, but the structure of submodules was not   clear. In this paper,  we  determine  all submodules  of them, and  we   show  the corresponding
irreducible quotient modules which were exactly defined in \cite{CG}.  
\vs{5pt}

\ni{\bf Keywords:}
N=1 Heisenberg-Virasoro superalgebra, non-weight module, irreducible quotient    module.}

\ni{\it Mathematics Subject Classification (2020):} 17B10, 17B65, 17B68.}
\parskip .001 truein\baselineskip 6pt \lineskip 6pt

\tableofcontents
\section{Introduction}

Throughout this paper,  we use  $\C, \C^*$ and $\Z$    to denote the sets of complex numbers, nonzero complex numbers and integers, respectively.
   All vector superspaces
(resp. superalgebras, supermodules)   and
 spaces (resp. algebras, modules)    are considered to be over
$\C$, and all modules for Lie superalgebras considered  are $\Z_2$-graded.  For a Lie (super)algebra $\mathfrak{L}$, we denote by $U(\mathfrak{L})$
the universal enveloping algebra.

Investigating non-weight modules is a fundamental task in representation theory of Lie (super)algebras. With the   development of Lie (super)algebras,
a family  of  important  non-weight modules   called    $U({\mathfrak{h}})$-free modules  were introduced and studied. They are closely related to the intermediate  series   modules and  algebraic D-modules. 
Note that ${\mathfrak{h}}$ can be considered as   various degree-$0$ subalgebras, especially as the Cartan  subalgebra.  
Those modules were first defined over the complex matrix algebra $\mathfrak{sl}_{n+1}$   by Nilsson  in \cite{8}. At the same time,
Tan and Zhao obtained  them  by a very different method in   the reference  \cite{9}.  From then on, the  $U({\mathfrak{h}})$-free modules  were widely investigated over all kinds of  Lie (super)algebras
 (see, e.g., \cite{N2,LZ,XZ,TZ,CG,HCS,YYX1,15,CDL1,CDL,YYX3,DL}).

The N=1 Heisenberg-Virasoro superalgebra (also called a kind of superconformal current
algebra) arises in the context of mathematical
physics and theoretical physics (see \cite{2}), which can be seen as   a supersymmetric extension of the 
Heisenberg–Virasoro algebra. It can be realized from Balinsky–Novikov
superalgebras, which construct local translation invariant Lie superalgebras of vector valued functions on the line (see
\cite{PB}). It   can also be realized in \cite{GMP} by    studying    the first cohomology groups of centerless
N=1 super Virasoro algebras with coefficients in their tensor density modules. 
 Recently, the representation theory of the N=1 Heisenberg-Virasoro superalgebra has been widely studied, such as Verma modules in \cite{4,5}, cuspidal modules in \cite{LL,6},  smooth modules in  \cite{7} and $U(\mathbb{C}L_0)$-free modules 
 of rank 2  in  \cite{DL}, and so on. 
The aim of this paper is to construct  and study a class  of non-weight   modules over the N=1 Heisenberg-Virasoro superalgebra, which are free of rank 2 when regarded as modules over $\mathfrak{h}=\mathbb{C}L_0\oplus\mathbb{C}H_0$. Here,  $\mathfrak{h}$ is a degree-0 subalgebra, but it is not the Cartan subalgebra.  Based on  the idea in \cite{YYX3},  we precisely determine all submodules for those reducible non-weight modules.
 In particular, we provide a simple calculation method by using  Proposition \ref{pro2.21} and some  linear operators.
Applying the   results of $\mathfrak{g}$,   the non-weight  modules over   some subalgebras  are defined and studied.
This is also our motivation for giving this work. The list of  the main statements  of this paper is as follows.
 $$\begin{tabular}{l*{3}c}
    \hline
      Lie      &  Determine      &  Irreducible    & Main 
     \\            
          (super)algebra           & all submodules & quotient module  &   conclusions   \\\hline
    $\mathfrak{g}$      &  Yes         & $\Phi(\lambda, \beta,  b)$    &  Theorems \ref{th44555}, \ref{le4.366}, \ref{th664}  \\\hline
           $\mathfrak{hv}$             & Yes    & $\Phi_{\mathfrak{hv}}(\lambda, \beta, b)$  & Propositions \ref{le5.1}, \ref{lemm5.2} \\            
           $\mathfrak{ns}$           & Yes  & $\Phi_\mathfrak{ns} (\lambda,   b)$  & Propositions \ref{lem5.5}, \ref{lemma5.666}  \\
           $\mathfrak{fv}$                & No        &  ---     & Proposition \ref{lemm5.9999}
           \\
          $\mathfrak{hc}$              & No       & ---     & Proposition \ref{lemma51111} 
            \\
         \hline
  \end{tabular}$$

The rest of this paper is organized as follows.
In Section $2$, we first  recall  some notations,    basic   definitions and  known results.
After that, we  introduce a class of non-weight   $\mathfrak{g}$-modules  $\Omega(\lambda,
\beta)$ in Proposition \ref{pro2.21}.
In Section  $3$, we give a complete classification for       $U(\mathfrak{h})$-free modules of rank $2$  over  $\mathfrak{g}$  in Theorem \ref{th44555}, where $\mathfrak{h}$ contains the central element $H_0$.
In Section  $4$, we determine all submodules of   those reducible $\mathfrak{g}$-module $\Omega(\lambda,
\beta)$  in Theorem \ref{le4.366} by    Proposition \ref{pro2.21} and some useful linear operators.  
In Section 5, we  give all submodules of those   non-weight modules over the Heisenberg-Virasoro  algebra and  the Neveu-Schwarz algebra in Proposition \ref{le5.1}  and Proposition \ref{lem5.5}, respectively.
As far as we know,  those  Heisenberg-Virasoro modules have been defined in \cite{HCS}, but the authors only gave a   filtration of  submodules of them.
 Finally, we present a class of weight $\mathfrak{g}$-modules   by applying   the weighting functor to $\mathfrak{g}$-modules $\Omega(\lambda,
\beta)$ in Theorem \ref{th664}.
 
\section{A family of non-weight modules over $\mathfrak{g}$}
In this section,    we   construct a class of  non-weight  modules (also called  polynomial  modules) over the N=1 Heisenberg-Virasoro superalgebra. Let $\mathfrak{h}=\C L_{0}\oplus \C H_{0}$ be a degree-0 subalgebra. 
Those modules  are free of rank $2$ when regarded as $U(\mathfrak{h})$-modules.

\subsection{N=1 Heisenberg-Virasoro  superalgebra}
 Now we  recall the definition of the N=1 Heisenberg-Virasoro  superalgebra, which was studied in  
  \cite{2}. It was also independently introduced as a supersymmetric extension of
 the Beltrami algebra in  \cite{GSWX}.
\begin{defi}
   The N=1 Heisenberg-Virasoro  superalgebra  $\mathfrak{g}  = \text{span}\big\{L_m, H_m, G_p, Q_p,C_i\mid m \in \mathbb{Z},p\in\mathbb{Z}+\frac{1}{2},i=1,2,3\big\}$  is an  infinite-dimensional  Lie superalgebra, and
   satisfies  the following Lie super-brackets:
 \begin{equation}\label{2.1}
\aligned
&[L_m, L_n] = (m - n) L_{m+n}+\delta_{m+n,0} \frac{m^3 - m}{12} C_1, 
\\& [L_m, H_n] = -n H_{m+n}+\delta_{m+n,0} (m^2 + m)C_2, 
\\&[H_m, H_n] = m\delta_{m+n,0}C_3,
\  [L_m, G_p] = \left( \frac{m}{2} - p \right) G_{m+p},
\\&  [L_m, Q_p] = - \left( \frac{m}{2} + p \right) Q_{m+p},  
\  [H_m, G_p] = m Q_{m+p}, 
\\&
 [G_p, G_q] = 2 L_{p+q} + \delta_{p+q,0} \frac{4p^2 -1}{12} C_1,\\&
 [G_p, Q_q]= H_{p+q}+ (2p+1) \delta_{p+q,0} C_2,
 \\&  
 [Q_p, Q_q] =\delta_{p+q,0} C_3,
 \ 
 [H_m, Q_p] =[\mathfrak{g}, C_i] = 0,
\endaligned
\end{equation}
where $m, n \in \mathbb{Z}$, $p,q \in \mathbb{Z}+\frac{1}{2},i=1,2,3$.
\end{defi}
It is worth noting that $\mathfrak{g}$ contains four important subalgebras:
\begin{itemize}
\item[\rm(i)]  the   
Heisenberg-Virasoro algebra $\mathfrak{hv}=\mathrm{span}\{L_m,H_m,C_1,C_2,C_3\mid m\in\mathbb{Z}\}$;
\item[\rm(ii)] the   Neveu-Schwarz algebra $\mathfrak{ns}=\mathrm{span}\{L_m,G_p,C_1\mid m\in\mathbb{Z},p\in \mathbb{Z}+\frac{1}{2}\}$;
\item[\rm(iii)]  the  Fermion-Virasoro algebra $\mathfrak{fv}=\mathrm{span}\{L_m,Q_p,C_1,C_3\mid m\in\mathbb{Z},p\in \mathbb{Z}+\frac{1}{2}\}$;
\item[\rm(iv)]
the   Heisenberg-Clifford superalgebra $\mathfrak{hc}=\mathrm{span}\{H_m,Q_p,C_3\mid m\in\mathbb{Z},p\in \mathbb{Z}+\frac{1}{2}\}$.
\end{itemize}
The $\mathbb{Z}_2$-graded Lie superalgebra $\mathfrak{g}=\mathfrak{g}_{\bar0}\oplus\mathfrak{g}_{\bar1}$, where $\mathfrak{g}_{\bar0}=\mathfrak{hv}$ and $\mathfrak{g}_{\bar1}=\mathrm{span}\{G_p,Q_p\mid p\in\mathbb{Z}+\frac{1}{2}\}$.  
Clearly, $\mathfrak{c} =\mathbb{C}H_0\oplus\mathbb{C}C_1\oplus\mathbb{C}C_2\oplus\mathbb{C}C_3$ is the center of $\mathfrak{g}$.
We denote by
$A=\mathbb{C}\left[t^{\frac{1}{2}}, t^{-\frac{1}{2}}, \xi\right]
$
the associative superalgebra of Laurent polynomials associated to an even, $t$, and an odd, $\xi$, formal variable.   
 Then  $\mathfrak{g}=\mathfrak{ns}\ltimes A$ can   be written as the super  derivation algebra of the form as
$$ L_m=-t^{m}(t\frac{d}{dt}+\frac{m}{2}\xi\frac{d}{d\xi}), G_p=t^{p}(\xi  \frac{d}{dt}- \frac{d}{d\xi}),H_m=t^m,Q_p=t^p\xi
$$
 for $m\in\mathbb{Z}$, $p\in\mathbb{Z}+\frac{1}{2}$.
\subsubsection{The known results}
In   the present paper, we always set $\lambda \in \mathbb{C}^*, \beta(y)=\sum_{i=0}^k\beta_iy^i \in \mathbb{C}[y]$.    Now we recall a family of non-weight modules over $\mathfrak{hv}$ constructed in \cite{HCS}, which are  closely  associated with the classification of $U(\mathfrak{h})$-free modules of rank 1.
    As a vector space, $\Psi_{\mathfrak{hv}}
(\lambda, \beta)= \mathbb{C}[x,y]$ is the polynomial algebra, and the $\mathfrak{hv}$-action is given by
\begin{eqnarray}\label{eq2.22}
L_m f(x,y) = \lambda^m (x+ m\beta(y)) f(x+m,y), \  H_m f(x,y) = \lambda^m yf(x+m,y),\ C_if(x,y)=0,\end{eqnarray}
where $f(x,y) \in \mathbb{C}[x,y], m \in \mathbb{Z},i=1,2,3$. 
The following results are given in \cite{HCS}.   
\begin{theo}\label{tho2.1}  Let $\lambda,\lambda^\prime\in\mathbb{C}^*$ and $\beta(y),\beta^\prime(y)\in\mathbb{C}[y]$.
\begin{itemize}
\item[\rm(1)] Then $\Psi_{\mathfrak{hv}}
(\lambda, \beta)$ is an $\mathfrak{hv}$-module;
\item[\rm(2)]
Then $\Psi_{\mathfrak{hv}}
(\lambda, \beta)\cong\Psi_{\mathfrak{hv}}
(\lambda^\prime, \beta^\prime)$ if and only if $\lambda=\lambda^\prime,\beta(y)=\beta^\prime(y)$;
\item[\rm(3)]
  Assume that there exists an $\mathfrak{hv}$-module $M^\prime$ such that it is a  
$U(\mathfrak{h})$-free module of rank 1. Then    $M^\prime\cong \Psi_{\mathfrak{hv}}
(\lambda, \beta)$.
\end{itemize}
\end{theo}

Assume that $\mathfrak{L}$ is a Lie superalgebra. Let $P=P_{\bar0}\oplus P_{\bar1}$ be a  $\Z_2$-graded vector space.  Any element $v\in P_{\bar0}$
is said to be even, and any element $v\in P_{\bar1}$
is said to be odd.   If
$v$ is even, we set $|v|=\bar 0$, and if
$v$ is  odd, we  set $|v|=\bar 1$.  Any elements in  $P_{\bar0}$ or  $P_{\bar1}$ are called homogeneous, and  all  elements in superalgebras and modules are homogeneous unless
specified.
An $\mathfrak{L}$-module is a $\Z_2$-graded vector space $V$ together
with a bilinear map $\mathfrak{L}\times P\rightarrow P$, denoted $(a,v)\mapsto av$ such that
$$a(bv)-(-1)^{|a||b|}b(av)=[a,b]v
\quad \mathrm{and}\quad
 \mathfrak{L}_{\bar i} P_{\bar j}\subseteq  P_{\bar i+\bar j},$$
where $\bar i,\bar j\in\Z_2,a, b \in  \mathfrak{L}, v \in P$. Therefore, there is a parity-change functor $\Pi$ on the category of
$\mathfrak{L}$-modules to itself, that is to say, for any module
$P=P_{\bar0}\oplus P_{\bar1}$, we obtain $\Pi(P_{\bar 0})=P_{\bar 1}$ and $\Pi(P_{\bar 1})=P_{\bar0 }$.

\subsection{Non-weight $\mathfrak{g}$-modules}
Now  we  give a  precise construction of a  $\mathfrak{g}$-module structure
on $M=\mathbb{C}[x,y] \oplus\mathbb{C}[s,t]$ as follows.
\begin{prop}\label{pro2.21}
For $\lambda\in \C^*,\beta(y)\in\C[y],\beta(t)\in\mathbb{C}[t],f(x,y)\in\C[x,y],f(s,t)\in\C[s,t]$,
 the   $\mathfrak{g}$-action
  on $M$ is defined as follows
\begin{eqnarray}
 &&\label{2.3}  
 L_m f(x,y) = \lambda^m (x+m\beta(y)) f(x+m,y),  
 \\&&\label{2.6}L_m f(s,t)= \lambda^m (s+m(\beta(t)+\frac{1}{2})) f(s+m,t), 
 \\&& \label{2.4} H_m f(x,y)= \lambda^m yf(x+m,y),
 \\&&\label{2.7}H_m f(s,t)= \lambda^m tf(s+m,t),
 \\&& \label{2.5}G_p f(x,y)= \lambda^{p-\frac{1}{2}} f(s+p,t),
 \\&& \label{2.8}G_p f(s,t)= \lambda^{p+\frac{1}{2}} (x + 2p\beta(y)) f(x+p,y),
 \\&&\label{2.9}  Q_p f(s,t)= \lambda^{p+\frac{1}{2}} yf(x+p,y),
 \\&& \label{qwe2.10}    Q_p f(x,y)=C_if(x,y) =C_if(s,t)=0 
 \end{eqnarray}
for $m\in\Z,p\in\mathbb{Z}+\frac{1}{2},i=1,2,3.$
 Then $M$ is a   $\mathfrak{g}$-module under the actions  of \eqref{2.3}-\eqref{qwe2.10}, which
  is  
free of rank $2$ as a $U(\mathfrak{h})$-module  and  is denoted by $\Omega(\lambda,
\beta)$.
\end{prop}
 \begin{proof} Since 
 $C_1,C_2$ and $C_3$
  ‌act‌ as zero on $M$, 
   we can ignore these central elements in this proof.
For any $m\in\mathbb{Z},p\in\mathbb{Z}+\frac{1}{2}$, according to \eqref{2.3}, \eqref{2.6}, \eqref{2.5} and \eqref{2.8}, we compute that 

 \begin{eqnarray*}
    &&[L_m, G_p] f(x,y) 
    \\
    &=& \lambda^{p-\frac{1}{2}} L_mf(s+p,t) - \lambda^m G_p(x+m\beta(y)) f(x+m,y) \\
    &=& \lambda^{m+p-\frac{1}{2}}   (s+m(\beta(t)+\frac{1}{2})) f(s+m+p,t) - \lambda^{m+p-\frac{1}{2}} (s+p+m\beta(t))   f(s+m+p,t)\\
    &=&(\frac{m}{2}-p)\lambda^{m+p-\frac{1}{2}}f(s+p+m,t)
  =( \frac{m}{2} - p ) G_{m+p}f(x,y)
    \end{eqnarray*}
    and
    \begin{eqnarray*}
    &&[L_m, G_p] f(s,t) 
    \\&=& \lambda^{p+\frac{1}{2}} L_m(x + 2p\beta(y)) f(x+p,y) -\lambda^m G_p(s+m(\beta(t)+\frac{1}{2})) f(s+m,t)  \\
    &=& \lambda^{m+p+\frac{1}{2}} (x+m\beta(y))  (x+m + 2p\beta(y)) f(x+m+p,y) \\
    &&-\lambda^{m+p+\frac{1}{2}} (x + 2p\beta(y)) (x+p+m(\beta(y)+\frac{1}{2})) f(x+m+p,y)  \\
    &=& \lambda^{m+p+\frac{1}{2}}  (\frac{m}{2}-p)(x+2(m+p)\beta(y)) f(x+m+p,y) \\
    &=&  \big( \frac{m}{2} - p \big) G_{m+p}f(s,t). 
    \end{eqnarray*}
 For any $m\in\mathbb{Z},p\in\mathbb{Z}+\frac{1}{2}$, it follows from \eqref{2.3},  \eqref{2.6}, \eqref{2.9} and \eqref{qwe2.10}  that we have 
    \begin{eqnarray*}
    &&[L_m, Q_p] f(x,y) = -\big( \frac{m}{2} + p \big) Q_{m+p}f(x,y)=0
    \end{eqnarray*}
    and
    \begin{eqnarray*}
    && [L_m, Q_p] f(s,t) 
    \\&=& \lambda^{p+\frac{1}{2}}  L_myf(x+p,y) -\lambda^m Q_p(s+m(\beta(t)+\frac{1}{2})) f(s+m,t)  
    \\
    &=&\lambda^{m+p+\frac{1}{2}}   (x+m\beta(y))yf(x+p+m,y)   -\lambda^{m+p+\frac{1}{2}} (x+p+m(\beta(y)+\frac{1}{2}))yf(x+p+m,y)  \\
    &=& -\lambda^{m+p+\frac{1}{2}}  (\frac{m}{2}+p) yf(x+p+m,y)= - \big( \frac{m}{2} + p \big) Q_{m+p}f(s,t). 
    \end{eqnarray*}
 For any $m\in\mathbb{Z},p\in\mathbb{Z}+\frac{1}{2}$, by \eqref{2.4}, \eqref{2.7}, \eqref{2.5} and \eqref{2.8}, we check
    \begin{eqnarray*}
    && [H_m, G_p] f(x,y)  =m Q_{m+p}f(x,y)=0
    \end{eqnarray*}
    and
    \begin{eqnarray*}
    && [H_m, G_p] f(s,t)
    \\
    &=&\lambda^{m+p+\frac{1}{2}}  (x+m+2p\beta{(y)})yf(x+m+p,y) - \lambda^{m+p+\frac{1}{2}}  (x+2p\beta{(y)})yf(x+m+p,y)\\
    &=&m Q_{m+p}f(s,t).
    \end{eqnarray*}  
 For any $p,q\in\mathbb{Z}+\frac{1}{2}$,   based on  \eqref{2.3}, \eqref{2.6},  \eqref{2.5} and \eqref{2.8}, we confirm that 
\begin{eqnarray*}
&&[G_p, G_q] f(x, y)  \\
&=& \lambda^{p+q}(x + 2p\beta(y)) f(x+p+q,y) + \lambda^{p+q}(x + 2q\beta(y)) f(x+p+q,y) \\
&=& 2\lambda^{p+q}(x + (p+q)\beta(y)) f(x+p+q,y)= 2 L_{p+q}f(x, y)
\end{eqnarray*}
and
\begin{eqnarray*}
&&[G_p, G_q] f(s, t)  
\\
&=& \lambda^{p+q} (s+p + 2q\beta(t)) f(s+p+q,t)+\lambda^{p+q} (s+q + 2p\beta(t)) f(s+p+q,t)\\
&=&2\lambda^{p+q} (s+(p+q)(\beta(t)+\frac{1}{2})) f(s+p+q,t)= 2 L_{p+q}f(s, t).
\end{eqnarray*}
For any $p,q\in\mathbb{Z}+\frac{1}{2}$, from \eqref{2.4}-\eqref{qwe2.10}, we have
    \begin{eqnarray*}
   && [G_p, Q_q] f(x,y) 
    =\lambda^{p+q} yf(x+p+q,y)
    =H_{p+q}f(x,y),
    \\&&
 [G_p, Q_q ]f(s,t) = \lambda^{p+q}tf(s+p+q,t)
=H_{p+q}f(s,t).
    \end{eqnarray*}  
 For any $m\in\mathbb{Z},p,q\in\mathbb{Z}+\frac{1}{2}$,
by using  \eqref{2.4},  \eqref{2.7},     \eqref{2.9} and \eqref{qwe2.10}, it is easy to get 
\begin{eqnarray*}
    [H_m, Q_p] f(x,y)=
    [H_m, Q_p] f(s,t)=[Q_p, Q_q] f(x,y) =[Q_p, Q_q] f(s,t) =0.
\end{eqnarray*}
Moreover, all the rest of the relations can be directly obtained by Theorem \ref{tho2.1}. This completes the proof. 
 \end{proof}
 \begin{rema} We note that $\mathfrak{g}$-module $\Omega(\lambda,\beta)$  is not irreducible,  and  
we remark   $\Omega(\lambda,
\beta)=(\Omega(\lambda,
\beta))_{\bar0}\oplus(\Omega(\lambda,
\beta))_{\bar1}$,
where $(\Omega(\lambda,
\beta))_{\bar0}=\mathbb{C}[x,y]$ and $(\Omega(\lambda,
\beta))_{\bar1}=\mathbb{C}[s,t]$.
 \end{rema}

\section{ Classification of  $U(\mathfrak{h})$-free modules of rank $2$}
In this section, we determine a complete classification of $U(\mathfrak{h})$-free modules of rank $2$ over the N=1 Heisenberg-Virasoro superalgebra.
The following result appeared in \cite{YYX3}.
\begin{lemm}\label{le311}  
  Let $\mathfrak{l}=\mathfrak{l}_{\bar0}\oplus\mathfrak{l}_{\bar1}$ be a Lie superalgebra. Let $\mathfrak{h}$ be a degree-0 subalgebra of $\mathfrak{l}$ with
$\mathfrak{h}\subseteq \mathfrak{l}_{\bar0}$, and $[\mathfrak{l}_{\bar1},\mathfrak{l}_{\bar1}]=\mathfrak{l}_{\bar0}.$ Then there do not exist $\mathfrak{l}$-modules which are free of rank 1 as
$U(\mathfrak{h})$-modules.  
\end{lemm}
Clearly,  $\mathfrak{g}$ has a degree-0 subalgebra  
$\mathfrak{h}= \mathbb{C} L_0 \oplus\mathbb{C} H_0 \subseteq\mathfrak{g}_{\bar0}$, and  $\mathfrak{g}$ is generated by odd elements $\{G_p,Q_p\mid p\in\mathbb{Z}+\frac{1}{2}\}$. From  Lemma \ref{le311},  we  see that there
do not exist $\mathfrak{g}$-modules which are free over $U(\mathfrak{h})$  of rank 1. Hence, we    classify the   $U(\mathfrak{h})$-free modules of rank 2 over    $\mathfrak{g}$ in this paper.

Let  $M=M_{\bar 0}\oplus M_{\bar1}$ be a  $\mathfrak{g}$-module such that it is    free of  rank $2$ as a $U(\mathfrak{h})$-module with two homogeneous
basis elements $u$  and $v$.
  Obviously, $u$ and $v$ have different parities. Let $u=\mathbf{1}_{\bar0}\in M_{\bar0}$ and 
$v=\mathbf{1}_{\bar1}\in M_{\bar1}$. From   \eqref{2.1},  we have
$L_0H_0=H_0L_0$.
So $$M=U(\mathfrak{h}) \mathbf{1}_{\bar0}\oplus  U(\mathfrak{h})\mathbf{1}_{\bar1}=\C[L_0,H_0] \mathbf{1}_{\bar0}\oplus  \C[L_0,H_0]\mathbf{1}_{\bar1}$$
with $M_{\bar 0}=\C[L_0,H_0] \mathbf{1}_{\bar0}$ and $M_{\bar1}=\C[L_0,H_0]\mathbf{1}_{\bar1}$.

By definition, it is easy to see the following lemma.

\begin{lemm}\label{lem3.2}
For any $m\in \mathbb{Z},p\in\mathbb{Z}+\frac{1}{2}, i \in \mathbb{N}$, we have
 \begin{eqnarray*}
  &&  L_m L_0^i = (L_0 + m)^i L_m, \ L_m H_0^i = H_0^i L_m,\\
  &&  H_m L_0^i = (L_0 + m)^i H_m, \ H_m H_0^i = H_0^i H_m,\\
  &&  G_p L_0^i = (L_0 + p)^i G_p, \ G_p H_0^i = H_0^i G_p,\\
  &&  Q_p L_0^i = (L_0 + p)^i Q_p, \ Q_p H_0^i = H_0^i Q_p.
 \end{eqnarray*}
\end{lemm}

\textbf{Proof.} According to     \eqref{2.1}, for any $m\in\mathbb{Z},p\in\mathbb{Z}+\frac{1}{2}$, it is easy to check that
\begin{eqnarray*}
  &&  L_m L_0 = (L_0 + m) L_m, \ L_m H_0 = H_0 L_m, \\
  &&  H_m L_0 = (L_0 + m) H_m, \ H_m H_0 = H_0 H_m, \\
  &&  G_p L_0 = (L_0 + p) G_p, \ G_p H_0 = H_0 G_p, \\
  &&  Q_p L_0 = (L_0 +p) Q_p, \ Q_p H_0 = H_0 Q_p.
\end{eqnarray*}

Then the lemma can be proven by induction on the degree of $L_0$ and $H_0$. \hfill $\square$

Since the even part of the N=1 Heisenberg-Virasoro superalgebra $\mathfrak{g}_{\bar0}$ is isomorphic to
the  Heisenberg–Virasoro algebra $\mathfrak{hv}$, we can regard both $M_{\bar0}$  and $M_{\bar1}$ as $\mathfrak{hv}$-modules.
According to Theorem \ref{tho2.1}, there exist $\lambda,\hat{\lambda}\in\mathbb{C}^*$,$\beta(H_0)=\sum_{i=0}\beta_iH_0^i\in \mathbb{C}[H_0],\hat{\beta}(H_0)=\sum_{i=0}\hat{\beta}_iH_0^i\in \mathbb{C}[H_0]$, $\beta_i,\hat{\beta}_i\in\mathbb{C}$  and $f(L_0,H_0)\in\mathbb{C}[L_0,H_0]$ 
such that
\begin{eqnarray*}
 &&  L_m f(L_0,H_0)\mathbf{1}_{\bar0} = \lambda^m (L_0+ m\beta(H_0)) f(L_0+m,H_0)\mathbf{1}_{\bar0}, 
 \\&&H_m f(L_0,H_0)\mathbf{1}_{\bar0} = \lambda^m H_0f(L_0+m,H_0)\mathbf{1}_{\bar0},
 \\&&L_m f(L_0,H_0) \mathbf{1}_{\bar1}= \hat{\lambda}^m (L_0+ m\hat{\beta}(H_0)) f(L_0+m,H_0)\mathbf{1}_{\bar1},  
 \\&&H_m f(L_0,H_0) \mathbf{1}_{\bar1}= \hat{\lambda}^m H_0f(L_0+m,H_0)\mathbf{1}_{\bar1},
 \\&&C_i f(L_0,H_0) \mathbf{1}_{\bar0}= C_i f(L_0,H_0) \mathbf{1}_{\bar1}=0,
\end{eqnarray*}
where $m,n\in\mathbb{Z},i=1,2,3.$
\begin{lemm}\label{lem3.22}
Keep the same notations as above. Then  $\lambda=\hat{\lambda}$ and there exists   
$c_p,\alpha_p\in\mathbb{C}^*$ such that one
of the following two cases occurs.
\begin{itemize}
\item[\rm(1)]$G_p  \mathbf{1}_{\bar{0}}=\frac{\lambda^{2p}}{c_p}(L_0+ 2p(\beta(H_0)-\frac{1}{2}))\mathbf{1}_{\bar{1}}$,  $G_p  \mathbf{1}_{\bar{1}}=c_p\mathbf{1}_{\bar{0}}$ and 
$\hat{\beta}(H_0)=\beta(H_0)-\frac{1}{2}$;
\item[\rm(2)]
$G_p  \mathbf{1}_{\bar{0}}=\alpha_p\mathbf{1}_{\bar{1}}$, $G_p  \mathbf{1}_{\bar{1}}=\frac{\lambda^{2p}}{\alpha_p}(L_0+ 2p\beta(H_0))\mathbf{1}_{\bar{0}}$ and  $\hat{\beta}(H_0)=\beta(H_0)+\frac{1}{2}$.
\end{itemize}
\end{lemm}

\begin{proof}
To prove this, we suppose that $G_p  \mathbf{1}_{\bar{0}}=f_p(L_0,H_0)\mathbf{1}_{\bar{1}}$ and $G_p  \mathbf{1}_{\bar{1}}=g_p(L_0,H_0)\mathbf{1}_{\bar{0}}$ for any $p\in\mathbb{Z}+\frac{1}{2}$.
For any $p, q \in \mathbb{Z}+\frac{1}{2}$, we have
\begin{eqnarray*}
&&G_p G_q  \mathbf{1}_{\bar{0}} + G_q G_p  \mathbf{1}_{\bar{0}} \\
&=& G_p f_q(L_0,H_0)  \mathbf{1}_{\bar{1}} + G_q f_p(L_0,H_0)  \mathbf{1}_{\bar{1}} \\
&=& f_q(L_0+p,H_0) G_p   \mathbf{1}_{\bar{1}} + f_p(L_0+q,H_0)G_q  \mathbf{1}_{\bar{1}} \\
&=& f_q(L_0+p,H_0) g_p(L_0,H_0)   \mathbf{1}_{\bar{0}} + f_p(L_0+q,H_0)g_q(L_0,H_0)  \mathbf{1}_{\bar{0}}.
\end{eqnarray*}
Using $[G_p, G_q]  \mathbf{1}_{\bar{0}}=2L_{p+q}\mathbf{1}_{\bar{0}}$, we obtain
\begin{eqnarray}\label{3.3}
2\lambda^{p+q}(L_0+(p+q)\sum_{i=0}\beta_i H_0^i)= f_q(L_0+p,H_0) g_p(L_0,H_0)    + f_p(L_0+q,H_0)g_q(L_0,H_0). 
\end{eqnarray}
For any $p,q\in\mathbb{Z}+\frac{1}{2}$,  we compute  
\begin{eqnarray*}
&&2\hat{\lambda}^{p+q}(L_0+(p+q)\sum_{i=0}\hat{\beta}_i H_0^i)\mathbf{1}_{\bar{1}}
\\&=& 2L_{p+q} \mathbf{1}_{\bar{1}}
 =[G_p, G_q] \mathbf{1}_{\bar{1}}\\
&=& G_p g_q(L_0,H_0)  \mathbf{1}_{\bar{0}} + G_q g_p(L_0,H_0)  \mathbf{1}_{\bar{0}} \\
&=& g_q(L_0+p,H_0) G_p   \mathbf{1}_{\bar{0}} + g_p(L_0+q,H_0)G_q  \mathbf{1}_{\bar{0}} \\
&=& g_q(L_0+p,H_0) f_p(L_0,H_0)   \mathbf{1}_{\bar{1}} + g_p(L_0+q,H_0)f_q(L_0,H_0)  \mathbf{1}_{\bar{1}}, 
\end{eqnarray*}
which forces
\begin{eqnarray}\label{3.4}
2\hat{\lambda}^{p+q}(L_0+(p+q)\sum_{i=0}\hat{\beta}_i H_0^i)= g_q(L_0+p,H_0) f_p(L_0,H_0)   + g_p(L_0+q,H_0)f_q(L_0,H_0).
\end{eqnarray}
Taking  $p=q$ in \eqref{3.3} and \eqref{3.4}, we conclude 
\begin{eqnarray}\label{3.5}
\lambda^{2p}(L_0+2p\sum_{i=0}\beta_i H_0^i)= f_p(L_0+p,H_0) g_p(L_0,H_0) 
\end{eqnarray}
and 
\begin{eqnarray}\label{3.6}
\hat{\lambda}^{2p}(L_0+2p\sum_{i=0}\hat{\beta}_i H_0^i)= g_p(L_0+p,H_0) f_p(L_0,H_0).
\end{eqnarray}
Then from \eqref{3.5}, we deduce  
\begin{eqnarray} 
&& \nonumber \label{eq3.9}f_p(L_0+p,H_0)=a_p(L_0+p)+b_p(H_0),\ g_p(L_0,H_0)=c_p\neq0,
\\&&a_pc_p=\lambda^{2p},(a_pp+b_p(H_0)) c_p= 2p\lambda^{2p}\sum_{i=0}\beta_i H_0^i
\end{eqnarray}
or 
\begin{eqnarray} 
&&\nonumber \label{eq3.10}f_p(L_0+p,H_0)=\alpha_p\neq0,g_p(L_0,H_0)=\beta_pL_0+\gamma_p(H_0),
\\&&\alpha_p\beta_p={\lambda}^{2p},\ \alpha_p\gamma_p(H_0)= 2p{\lambda}^{2p}\sum_{i=0}{\beta}_i H_0^i. 
\end{eqnarray}
Inserting \eqref{eq3.9} and \eqref{eq3.10}  into
\eqref{3.6},    we respectively confirm that  
$$\lambda=\hat{\lambda}, \hat{\beta}(H_0)={\beta}(H_0)-\frac{1}{2},a_p= \frac{\lambda^{2p}}{c_p},b_p(H_0)=\frac{2p\lambda^{2p}}{c_p}(\beta(H_0)-\frac{1}{2})$$
and 
$$\lambda=\hat{\lambda},\hat{\beta}(H_0)={\beta}(H_0)+\frac{1}{2},\beta_p=\frac{\lambda^{2p}}{\alpha_p},\gamma_p(H_0)=\frac{2p\lambda^{2p}}{\alpha_p}\beta(H_0).$$
The lemma holds.
\end{proof}
From Lemma \ref{lem3.22}, up to a parity, we can suppose $G_p  \mathbf{1}_{\bar{0}}=\alpha_p\mathbf{1}_{\bar{1}}, G_p  \mathbf{1}_{\bar{1}}=\frac{\lambda^{2p}}{\alpha_p}(L_0+ 2p\beta(H_0))\mathbf{1}_{\bar{0}}$ and  $\hat{\beta}(H_0)=\beta(H_0)+\frac{1}{2}$ without a loss of generality, where  $\alpha_p\in\mathbb{C}^*$.  Then we give the following statements.
\begin{lemm}\label{lem3.3}
For any $p\in\Z+\frac{1}{2}$, we obtain
\begin{itemize}
\item[\rm(1)]
$G_p  \mathbf{1}_{\bar{1}} = \lambda^{p+\frac{1}{2}} \big(L_0 + 2p\beta(H_0)\big)  \mathbf{1}_{\bar{0}}$ and 
$G_p  \mathbf{1}_{\bar{0}} = \lambda^{p-\frac{1}{2}}  \mathbf{1}_{\bar{1}}$;
\item[\rm(2)]
$Q_p  \mathbf{1}_{\bar{1}} =\lambda^{p+\frac{1}{2}}H_0\mathbf{1}_{\bar{1}}$ and $Q_p  \mathbf{1}_{\bar{0}} =0 $.
\end{itemize}
\end{lemm}

\begin{proof}
{\rm (1)}
For any $m\in\mathbb{Z},p \in \mathbb{Z} + \frac{1}{2}$, we have 
\begin{eqnarray*}
&&(\frac{m}{2}-p)\frac{\lambda^{2(m+p)}}{\alpha_{m+p}}\big(L_0+2(m+p) \beta(H_0)\big)\mathbf{1}_{\bar{0}}
\\&=&[L_m, G_p] \mathbf{1}_{\bar{1}} 
\\&=& L_m G_p  \mathbf{1}_{\bar{1}} - G_p L_m \mathbf{1}_{\bar{1}} \\
&=& \frac{\lambda^{m+2p}}{\alpha_p}\big(L_0+m+2p\beta(H_0)\big)\big(L_0+m\beta(H_0)\big)\mathbf{1}_{\bar{0}}
\\&&-\frac{\lambda^{m+2p}}{\alpha_p}\big(L_0+p+m(\beta(H_0)+\frac{1}{2})\big)\big(L_0+2p\beta(H_0)\big)\mathbf{1}_{\bar{0}}
\\&=& (\frac{m}{2}-p)\frac{\lambda^{m+2p}}{\alpha_p}\big(L_0+2(m+p)\beta(H_0)\big)\mathbf{1}_{\bar{0}},
\end{eqnarray*}
which implies 
    \begin{eqnarray}\label{3.11}
    (\frac{m}{2}-p)\frac{\lambda^{2(m+p)}}{\alpha_{m+p}}=(\frac{m}{2}-p)\frac{\lambda^{m+2p}}{\alpha_p}.
\end{eqnarray}
Setting $p=\frac{1}{2}$ in \eqref{3.11}, we obtain  $\alpha_{m+\frac{1}{2}}=\lambda^{m}\alpha_{\frac{1}{2}}$ for $m\neq1$. Taking $m=-1$ and $p=\frac{3}{2}$ in \eqref{3.11} again, we see that $\alpha_{\frac{3}{2}}=\lambda \alpha_{\frac{1}{2}}$. So, we always have 
$\alpha_{m+\frac{1}{2}}=\lambda^{m}\alpha_{\frac{1}{2}}$ for any $m\in\mathbb{Z}$.
Replacing $\alpha_{\frac{1}{2}}\mathbf{1}_{\bar{1}}$ by $\mathbf{1}_{\bar{1}}$,  we confirm that (1).

{\rm (2)}
For any $0\neq m\in\mathbb{Z}$ and $p\in\mathbb{Z}+\frac{1}{2}$,
by using $[H_m, G_{p-m}] \mathbf{1}_{\bar{1}} = m Q_p \mathbf{1}_{\bar{1}}$, one can check that 
\begin{eqnarray*} 
Q_p \mathbf{1}_{\bar{1}} &=& \frac{1}{m}[H_m, G_{p-m}] \mathbf{1}_{\bar{1}} \\
&=& \frac{1}{m}  \lambda^{p-m+\frac{1}{2}} \big(L_0+m + 2(p-m)\beta(H_0)\big) H_m\mathbf{1}_{\bar{0}} - \frac{1}{m} G_{p-m} \lambda^m H_0  \mathbf{1}_{\bar{1}}
\\&=&\lambda^{p+\frac{1}{2}}H_0\mathbf{1}_{\bar{1}}.
\end{eqnarray*}

For any $0\neq m\in\mathbb{Z}$ and $p\in\mathbb{Z}+\frac{1}{2}$,
by $[H_m, G_{p-m}] \mathbf{1}_{\bar{0}} = m Q_p \mathbf{1}_{\bar{0}}$, we know 
that 
\begin{eqnarray*}
Q_p \mathbf{1}_{\bar{0}} &=& \frac{1}{m}[H_m, G_{p-m}] \mathbf{1}_{\bar{0}} \\
&=& \frac{1}{m} H_m G_{p-m}  \mathbf{1}_{\bar{0}} - \frac{1}{m} G_{p-m} H_m  \mathbf{1}_{\bar{0}}=0.
\end{eqnarray*}
The lemma clears.

\end{proof}
 
Now we present the main result of this section, which gives a
complete classification of   $U(\mathfrak{h})$-free modules of rank 2 over $\mathfrak{g}$.

\begin{theo}\label{th44555}
Let $M$  be a $\mathfrak{g}$-module such that the restriction of $M$ as a $U(\mathfrak{h})$-module is free of rank $2$, where  $\mathfrak{h}=\mathbb{C}L_0\oplus\mathbb{C}H_0$.
 Then up to a
parity, we have $M\cong\Omega(\lambda, \beta)$ for some  $\lambda\in\mathbb{C}^*$ and $\beta(y)\in\mathbb{C}[y]$  with the  $\mathfrak{g}$-module structure defined as in \eqref{2.3}-\eqref{qwe2.10}.
\end{theo}
\begin{proof}
Let $L_0\mathbf{1}_{\bar{0}}=x\mathbf{1}_{\bar{0}},H_0\mathbf{1}_{\bar{0}}=y\mathbf{1}_{\bar{0}},L_0\mathbf{1}_{\bar{1}}=s\mathbf{1}_{\bar{1}},H_0\mathbf{1}_{\bar{1}}=t\mathbf{1}_{\bar{1}}$. Based on  Theorem \ref{tho2.1}  and Lemmas \ref{lem3.2}-\ref{lem3.3}, it is easy to check the result.
\end{proof}

\section{The submodule structure of $\mathfrak{g}$-module $\Omega(\lambda,\beta)$}
In this section,  all submodules of $\Omega(\lambda,\beta)$ are precisely determined for any
$\lambda\in\mathbb{C}^*,\beta(y)\in\mathbb{C}[y]$, which  is mainly inspired by   \cite{YYX3}.

The following result can be found in  Proposition 7  of \cite{LZ0}.
\begin{prop}\label{pro4.1}
Let $P$ be a vector space over $\C$ and $P_1$ be a subspace of $P$. Assume
 that  $v_{i}\in P$
and $f_{i}(t)\in\C[t]$ with $\mathrm{deg}\,f_{i}(t)=i$ for $i=1,2,\ldots,s.$
If $\sum_{i=1}^{s}f_{i}(m)v_{i}\in P_1$ for  $m\in\Z,$
then $v_{i}\in P_1$ for all $i\in\mathbb{Z}$.
\end{prop}

 \begin{theo}\label{theo}
    Let $\lambda,\lambda^\prime\in\mathbb{C}^*,\beta(y),\beta^\prime(y)\in\mathbb{C}[y]$. Then $\Omega(\lambda,\beta)
	\cong
\Omega(\lambda^\prime,\beta^\prime)$ as $\mathfrak{g}$-modules  if and only
if  $\lambda=\lambda^\prime,\beta(y)=\beta^\prime(y)$.
\end{theo}
\begin{proof} The sufficiency is clear. Suppose that $\phi: \Omega(\lambda,\beta) 
	\rightarrow
 \Omega(\lambda^\prime,\beta^\prime)$  is a  $\mathfrak{g}$-module isomorphism. By Theorem \ref{tho2.1} (2), 
 one can see  that $\lambda=\lambda^\prime$ and $\beta(y)=\beta^\prime(y)$.
\end{proof}

Now we show the main results of this section.  
 \begin{theo}\label{le4.366}
Let $\lambda\in\C^*$, $\beta(y)=\sum_{i=0}\beta_iy^i\in\mathbb{C}[y]$.
For  any $g(y)\in\mathbb{C}[y]$, we denote $R_g=g(y)\mathbb{C}[x,y]\oplus g(t)\mathbb{C}[s,t]\subseteq\Omega(\lambda,\beta) $  and
$S_g=g(y)(x\mathbb{C}[x,y]+y\mathbb{C}[x,y])\oplus g(t)\mathbb{C}[s,t] \subseteq\Omega(\lambda,\beta)$. Then the following statements hold.
\begin{itemize}
\item[\rm(1)] 
If $\beta_0\neq0$, the set $\big\{R_g \mid g(y)\in \mathbb{C}[y]\big\}$ exhausts all $\mathfrak{g}$-submodules of $\Omega(\lambda,\beta)$. Moreover, $R_g\subseteq  R_{\hat{g}}$  for
any $g(y), \hat{g}(y)\in \mathbb{C} [y]$  with $\hat{g}(y) \mid g(y)$.
\item[\rm(2)] If $\beta_0=0$, the set $\big\{R_g, S_g \mid g(y)\in \mathbb{C}[y]\big\}$ exhausts all $\mathfrak{g}$-submodules of $\Omega(\lambda,\beta)$. Moreover, $S_g \subseteq R_g\subseteq  R_{\hat{g}}$  for
any $g(y), \hat{g}(y)\in \mathbb{C} [y]$  with $\hat{g}(y) \mid g(y)$.
\item[\rm(3)]  For any nonzero $g(y) \in \mathbb{C}[y]$, the quotient $\Omega(\lambda,\beta)/R_g$ is a free $\mathbb{C}[L_0]$-module of rank $2 \mathrm{deg}g(y)$.  Any maximal submodule of $\Omega(\lambda,\beta)$ is of the form $R_g$ for some $g(y)\in\mathbb{C}[y]$ with $\mathrm{deg}g(y)=1$.  
\end{itemize}
  \end{theo}

  \begin{proof}
{\rm (1)} To prove this, suppose that  $F=
F_{\bar0}\oplus F_{\bar1}$ is a nonzero submodule of $\Omega(\lambda, \beta)=\mathbb{C}[x,y]\oplus \mathbb{C}[s,t]$. Take a nonzero $f(x, y)=\sum_{i=0}^kx^if_i(y)\in\mathbb{C} [x,y]$. 
 Consider $\beta(0)=\beta_0\neq0$. We  give the following  claim. 
\begin{clai}\label{claim21}
      $f(x,y)=\sum^k_{i=0}x^i
f_i(y)\in F_{\bar0}
\Longleftrightarrow
  f_i(y)\in F_{\bar0}$ for $i=0, 1,\ldots,k$.
\end{clai}
For any $m\in\mathbb{Z}$, we compute that  \begin{eqnarray}\label{4411.2}
&&\nonumber
\sum_{i=0}^k(\frac{1}{\beta_0}x+m)(x+m)^{i}
f_i(y)
\\&=&
\frac{1}{\beta_0}\lambda^{-m}\big(L_m-m\sum_{j=0}\beta_{j+1}H_0^{j}H_m\big)\sum^k_{i=0}x^i
f_i(y)
\in F_{\bar0}.
\end{eqnarray}
Using Proposition 
\ref{pro4.1} and considering the coefficient of $m^{k+1}$ in \eqref{4411.2}, we have     $f_k(y)\in F_{\bar0}.$
According to the recursive method, one can check that $f_i(y)\in F_{\bar0}$ for $i=0,1,\ldots,k$.
The claim clears.

Let $g(y)\in F_{\bar0}$ be a nonzero polynomial  such that $\mathrm{deg}_y(g(y))$ is minimal.
For any
$f(x,y)=\sum^k_{i=0}
x^i
f_i(y) \in F_{\bar0}$, it follows from the Claim \ref{claim21} that  $g(y)\mid f_i(y)$  for $i=0,1,\ldots,k$.
Then, $F_{\bar0}$ is generated by $g(y)$, namely, $F_{\bar0} = g(y)\mathbb{C}[x,y]$.   For any $u(s,t) \in \mathbb{C}[s,t]$, we have $g(t)u(s,t)=G_{\frac{1}{2}}g(y)u(x-\frac{1}{2},y)\in F_{\bar1}$.
Hence, $g(t)\mathbb{C}[s,t]\in F_{\bar1}$. On the other hand, choosing any $r(s,t)\in F_{\bar1}$,   we obtain  that 
\begin{eqnarray*}
   && (x+\beta_0)r(x,y)
   \\&=&(x+\beta(y))r(x,y)-\sum_{j=0}\beta_{j+1}H_0^{j}yr(x,y)
    \\&=&\lambda^{-1}\big(G_{\frac{1}{2}}-\sum_{j=0}\beta_{j+1}H_0^{j}Q_{\frac{1}{2}}\big)r(s-\frac{1}{2},t)\in g(y)\mathbb{C}[x,y].
\end{eqnarray*} 
That is to say, $F_{\bar1}\subseteq g(t)\mathbb{C}[s,t].$
So, $F=R_g$.

 {\rm (2)} Suppose that  $F=
F_{\bar0}\oplus F_{\bar1}$ is a nonzero submodule of $\Omega(\lambda, \beta)=\mathbb{C}[x,y]\oplus \mathbb{C}[s,t]$. Let $0\neq f(x, y)=\sum_{i=0}^kx^if_i(y)\in\mathbb{C} [x,y]$. Consider $\beta(0)=\beta_0=0$. We     present the   claim as follows. 
\begin{clai}\label{claim1}
      $f(x,y)=\sum^k_{i=0}x^i
f_i(y)\in F_{\bar0}
\Longleftrightarrow
   f_0(y), xf_i(y)\in F_{\bar0}$ for $i= 1,\ldots,k$.
\end{clai}
 For any $0\neq m\in\mathbb{Z}$, we have \begin{eqnarray}\label{eq4.1} 
 (m-1)x\sum^k_{i=0}(x+m)^if_i(y)=\lambda^{-m}\big(mG_{\frac{1}{2}}G_{m-\frac{1}{2}}-L_m\big)\sum^k_{i=0}x^i
f_i(y)\in F_{\bar0}.
\end{eqnarray}
By Proposition 
\ref{pro4.1} and considering the coefficient of $m^{k+1}$ in \eqref{eq4.1}, one can see that   $$xf_k(y)\in F_{\bar0}.$$
From   $L_0^{k-1}xf_k(y)$, one gets $x^kf_k(y)\in F_{\bar0}$. 
Therefore, $f(x,y)-x^kf_k(y)=\sum_{i=0}^{k-1}x^if_i(y)\in F_{\bar0}$.
By recursive method, it is easy to get $f_0(y),xf_i(y)\in F_{\bar0}$ for $i=1,\ldots,k$.
The claim holds.

Let $h(y),xg(y)\in F_{\bar0}$ be nonzero polynomials such that $\mathrm{deg}_y(h(y)),\mathrm{deg}_{y}(xg(y))$ are minimal.
For any
$f(x,y)=\sum^k_{i=0}
x_i
f_i(y) \in F_{\bar0}$, it follows from the Claim \ref{claim1} that $h(y)\mid f_0(y)$, $g(y)\mid f_i(y)$  for $i=1,\ldots,k$.
By $xh(y) = L_0h(y) \in F_{\bar0}$, we have $g(y)\mid h(y)$. For any $0\neq m\in\mathbb{Z}$, we check  $$ yg(y)=\frac{1}{m}(\lambda^{-m}H_m-H_0)xg(y) \in  F_{\bar0}.$$ Then  we get $h(y)
\mid yg(y)$. Therefore, $h(y) = c_1g(y)$  or $h(y) = c_2yg(y)$  for
some nonzero $c_1, c_2 \in\mathbb{C}$. We give the following discussions.
\begin{case}
   $h(y)=c_1g(y)$  for some 
   $c_1\in \mathbb{C}^*$. 
\end{case}
In this case, $F_{\bar0}$ is generated by $g(y)$, namely, $F_{\bar0} = g(y)\mathbb{C}[x,y]$.   For any $u(s,t) \in \mathbb{C}[s,t]$, we have $g(t)u(s,t)=G_{\frac{1}{2}}g(y)u(x-\frac{1}{2},y)\in F_{\bar1}$.
Hence, $g(t)\mathbb{C}[s,t]\in F_{\bar1}$. On the other hand, choosing any $r(s,t)\in F_{\bar1}$,    we check  that 
\begin{eqnarray*}
   && xr(x,y)
   \\&=&(x+\beta(y))r(x,y)-\sum_{j=0}\beta_{j+1}H_0^{j}yr(x,y)
    \\&=&\lambda^{-1}\big(G_{\frac{1}{2}}-\sum_{j=0}\beta_{j+1}H_0^{j}Q_{\frac{1}{2}}\big)r(s-\frac{1}{2},t)\in g(y)\mathbb{C}[x,y].
\end{eqnarray*} 
That is to say, $F_{\bar1}\subseteq g(t)\mathbb{C}[s,t].$
So, $F=R_g$.
\begin{case}
   $h(y)=c_2yg(y)$  for some 
   $c_2\in \mathbb{C}^*$. 
\end{case}
In this case, $F_{\bar0}$ is generated by $xg(y)$ and $yg(y)$, namely, $F_{\bar0}=g(y)\big(x\mathbb{C}[x,y]+y\mathbb{C}[x,y]\big)$. For any $u(s,t)\in \mathbb{C}[s,t]$, we obtain
\begin{eqnarray*}
    g(t)u(s,t)=\lambda^{-1}G_{\frac{3}{2}}g(y)xu(x-\frac{3}{2},y)-G_{\frac{1}{2}}g(y)xu(x-\frac{1}{2},y)\in F_{\bar1}.
\end{eqnarray*}
Hence, $g(t)\mathbb{C}[s,t]\subseteq F_{\bar1}$. On the other hand, taking any $r(s,t)\in F_{\bar1}$,  it is easy to get 
\begin{eqnarray*}
   && xr(x,y)
   \\&=&(x+\beta(y))r(x,y)-\sum_{j=0}\beta_{j+1}H_0^{j}yr(x,y)
    \\&=&\lambda^{-1}\big(G_{\frac{1}{2}}-\sum_{j=0}\beta_{j+1}H_0^{j}Q_{\frac{1}{2}}\big)r(s-\frac{1}{2},t) 
   \in g(y)(x\mathbb{C}[x,y]+y\mathbb{C}[x,y]).
\end{eqnarray*} 
This forces $F_{\bar1}\subseteq g(t)\mathbb{C}[s,t].$ Consequently, $F_{\bar1}= g(t)
\mathbb{C}[s,t]$ and $F=S_g.$

{\rm(3)} We can directly  check the results by (1) and (2).
  \end{proof}  
  \begin{rema}
      In Theorem \ref{le4.366} (2), if  $g(y)=y$, we see that the quotient module $\Omega(\lambda,\beta)/R_g$ has a submodule $x\mathbb{C}[x]\oplus \mathbb{C}[s]$ (see  \cite{DL}).  That is to say, $\Omega(\lambda,\beta)/R_g$ is not irreducible and $R_g$ is not the maximal submodule.
  \end{rema}
From Theorem \ref{le4.366}, we know that any irreducible quotient of the $\mathfrak{g}$-module $\Omega(\lambda,\beta)$ is of the form $\Omega(\lambda,\beta)/R_g$ for some
$g(y)=y-b$ with $b\in\mathbb{C}^*$ or $\beta(b)\neq0$, which is denoted by $\Phi(\lambda, \beta,  b)$. Here, if $b=0$, we have $\beta_0\neq0$. Then as a $\mathbb{Z}_2$-graded vector space, $\Phi(\lambda, \beta,  b)=\mathbb{C} [x]\oplus\mathbb{C} [s]$ with $\Phi(\lambda, \beta,  b)_{\bar0} = \mathbb{C} [x]$ and $\Phi(\lambda, \beta,  b)_{\bar1} = \mathbb{C}[s]$. By Theorem \ref{le4.366} and the definition of $\Omega(\lambda,\beta)$, 
it is clear that   the $\mathfrak{g}$-module structure of $\Phi(\lambda, \beta,  b)$ is presented as follows
\begin{eqnarray}
 &&\label{4.3} L_m f(x) = \lambda^m (x+m\beta(b)) f(x+m), 
 \\&&L_m f(s)= \lambda^m (s+m(\beta(b)+\frac{1}{2})) f(s+m), 
 \\&& H_m f(x)= \lambda^m bf(x+m),
  \\&& H_m f(s)= \lambda^m bf(s+m),
 \\&& G_p f(x)= \lambda^{p-\frac{1}{2}} f(s+p),
 \\&&G_p f(s)= \lambda^{p+\frac{1}{2}} (x + 2p\beta(b)) f(x+p),
 \\&&  Q_p f(s)= \lambda^{p+\frac{1}{2}} bf(x+p), 
 \\&&\label{4.9}  Q_p f(x)=C_if(x)=C_if(s)= 0
\end{eqnarray}
for $m\in\Z,p\in\mathbb{Z}+\frac{1}{2},i=1,2,3.$
Note that 
the $\mathfrak{g}$-module $\Phi(\lambda, \beta,  b)$   is irreducible if and only if $\beta(b)\neq0$ or $b\neq0$ (see  \cite{DL}). For $\lambda^\prime\in\mathbb{C}^*,  \beta^\prime(b^\prime),b^\prime\in\mathbb{C}$,  as $\mathfrak{g}$-modules $\Phi(\lambda, \beta,  b)\cong\Phi(\lambda^\prime, \beta^\prime,  b^\prime)$ if and only if $\lambda=\lambda^\prime, \beta(b)=\beta^\prime(b^\prime),  b=b^\prime$ (see  \cite{DL}).    From  Theorem \ref{le4.366}, we get that $S_g=x\mathbb{C}[x]\oplus \mathbb{C}[s]$  is a submodule if $\beta(b)=b=0$.
\begin{lemm}\label{lemm4.4}
    Keep the notations as Theorem \ref{le4.366}. 
  Let $\lambda\in \mathbb{C}^*,\beta(y)\in\mathbb{C}[y]$. Assume $g(y), g^\prime(y)\in\mathbb{C}[y]$ with
$g(y)=(y-b)g^\prime(y)$. Then as a  $\mathfrak{g}$-module $R_{g^\prime}/R_g \cong \Phi(\lambda,\beta,b)$. In particular, if $(b,\beta(b))\neq (0,0)$, the quotient $R_{g^\prime}/R_g$ is irreducible.
\end{lemm}
\begin{proof}
    As a $\mathbb{Z}_2$-graded vector space,
$$R_{g^\prime}/R_g = (R_{g^\prime}/R_g)_{\bar0}\oplus (R_{g^\prime}/R_g)_{\bar1}=g^\prime(y)
\mathbb{C}[x]\oplus g^\prime(t)\mathbb{C}[s].$$
We define the following linear mapping
  \begin{eqnarray*}
\psi: \Phi(\lambda,\beta,b)&\longrightarrow& R_{g^\prime}/R_g
\\ f(x)&\longmapsto&g^\prime(y)f(x),
\\  f(s)&\longmapsto& g^\prime(t)f(s).
\end{eqnarray*}
From \eqref{2.3}-\eqref{qwe2.10} and \eqref{4.3}-\eqref{4.9}, we check that $\psi$ is a $\mathfrak{g}$-module isomorphism.
\end{proof}
\begin{coro}
   Let $\lambda \in \mathbb{C}^*,\beta(y)\in\mathbb{C}[y]$. Assume that  $g(y)$ is a monic polynomial of degree $n$,    $\alpha_1, \ldots, \alpha_n$ are all roots of $g(y)$ in $\mathbb{C}$ (counting the multiplicity) and assume $(\alpha_i,\beta(\alpha_i))\neq(0,0)$ for $i=1,\ldots,n$. Denote $g_i(y) = (y - \alpha_1)(y -\alpha_2) \cdots (y -\alpha_i)$ and $R_i = R_{g_i}/R_g$ for $i = 1, \ldots, n$. Then
\[
\Omega(\lambda, \beta)/R_g \supset R_1 \supset R_2 \supset \cdots \supset R_{n-1} \supset R_n = 0
\]
is a decomposition series of the  quotient module $\Omega(\lambda, \beta)/R_g$. Moreover, we have the following irreducible quotient module
\[
(\Omega(\lambda, \beta)/R_g)/R_1 \cong \Phi(\lambda, \beta, \alpha_1), \quad R_i/R_{i+1} \cong \Phi(\lambda, \beta, \alpha_{i+1}), \quad 1 \leq i \leq n-1.
\]
We also have $\Phi(\lambda, \beta, \alpha_i) \cong \Phi(\lambda, \beta, \alpha_{j})$ if and only if $\alpha_i=\alpha_j$ for $i,j=1,\ldots,n.$
\end{coro}

\begin{proof}
It follows from Lemma \ref{lemm4.4} and    isomorphisms $(\Omega(\lambda, \beta)/R_g)/(R_{g_1}/R_g) \cong \Omega(\lambda, \beta)/R_{g_1}$ and  
$
(R_{g_i}/R_g)/(R_{g_{i+1}}/R_g) \cong R_{g_i}/R_{g_{i+1}}$ for $ 1 \leq i \leq n-1$
that we have the assertion.
\end{proof}

\section{Modules over some subalgebras of $\mathfrak{g}$}
In this section, we investigate the non-weight modules over the four  subalgebras of $\mathfrak{g}$, namely, the
Heisenberg-Virasoro algebra, the Neveu-Schwarz algebra, the Fermion-Virasoro algebra and the
Heisenberg-Clifford superalgebra. 
\subsection{Modules over the Heisenberg-Virasoro algebra}\label{sec5.1}
The $\mathfrak{hv}$-modules $\Psi_{\mathfrak{hv}}(\lambda,\beta)$ were constructed and studied in \cite{HCS}. The authors only   gave a   filtration   for those $\mathfrak{hv}$-modules.
In this subsection, we will determine  all submodules of $\mathfrak{hv}$-modules $\Psi_{\mathfrak{hv}}(\lambda,\beta)$  for any
$\lambda\in\mathbb{C}^*,\beta(y)\in\mathbb{C}[y]$. 
 \begin{prop}\label{le5.1}
     Assume that $\lambda\in\mathbb{C}^*,\beta(y)\in\C[y]$. 
For  any $g(y)\in\mathbb{C}[y]$, let $R_g^{\mathfrak{hv}}=g(y)\mathbb{C}[x,y]\subseteq\Psi_{\mathfrak{hv}}(\lambda,\beta) $  and
$S_g^{\mathfrak{hv}}=g(y)(x\mathbb{C}[x,y]+y\mathbb{C}[x,y]) \subseteq\Psi_{\mathfrak{hv}}(\lambda,\beta)$. Then we have the following results.
\begin{itemize}
\item[\rm(1)] If $\beta_0\neq0$, the set $\Big\{R_g^{\mathfrak{hv}} \mid g(y)\in \mathbb{C}[y]\Big\}$ exhausts all $\mathfrak{hv}$-submodules of $\Psi_{\mathfrak{hv}}(\lambda,\beta)$. Moreover, $R_g^{\mathfrak{hv}}\subseteq  R_{\hat{g}}^{\mathfrak{hv}}$  for
any $g(y), \hat{g}(y)\in \mathbb{C} [y]$  with $\hat{g}(y) \mid g(y)$.
If $\beta_0=0$, the set $\Big\{R_g^{\mathfrak{hv}}, S_g^{\mathfrak{hv}} \mid g(y)\in \mathbb{C}[y]\Big\}$ exhausts all $\mathfrak{hv}$-submodules of $\Psi_{\mathfrak{hv}}(\lambda,\beta)$. Moreover, $S_g^{\mathfrak{hv}} \subseteq R_g^{\mathfrak{hv}}\subseteq  R_{\hat{g}}^{\mathfrak{hv}}$  for
any $g(y), \hat{g}(y)\in \mathbb{C} [y]$  with $\hat{g}(y) \mid g(y)$.
\item[\rm(2)] For any nonzero $g(y) \in \mathbb{C}[y]$, the quotient module $\Psi_{\mathfrak{hv}}(\lambda,\beta)/R_g^{\mathfrak{hv}}$ is a free $\mathbb{C}[L_0]$-module of rank $\mathrm{deg}g(y)$. Any maximal submodule of $\Psi_{\mathfrak{hv}}(\lambda,\beta)$ is of the form $R_g^{\mathfrak{hv}}$ for some $g(y)\in\mathbb{C}[y]$ with $\mathrm{deg}g(y)=1$.  
\end{itemize}
  \end{prop}
  \begin{proof}
  {\rm (1)}
  If $\beta_0\neq0$,
  by  the similar computation in the proof  of Theorem \ref{le4.366}, we can check the  statement.  So,  we omit further details. 
  
  In the following proof, we only consider $\beta_0=0$.
Suppose that  $F^{\mathfrak{hv}}$ is a nonzero submodule of $\Psi_{\mathfrak{hv}}(\lambda, \beta)=\mathbb{C}[x,y]$. Choose  a nonzero element $f(x, y)=\sum_{i=0}^kx^if_i(y)\in\mathbb{C} [x,y]$.
We first present the  claim as follows.  
\begin{clai}\label{claim2}
      $f(x,y)=\sum^k_{i=0}x^i
f_i(y)\in F^{\mathfrak{hv}}
\Longleftrightarrow
   f_0(y), xf_i(y)\in F^{\mathfrak{hv}}$ for $i= 1,\ldots,k$.
\end{clai}
For any $m\in\mathbb{Z}$,  we confirm that 
 \begin{eqnarray}\label{55.2}
&&\sum_{i=0}^kx(x+m)^{i}
f_i(y)=\lambda^{-m}(L_m-m\sum_{j=0}\beta_{j+1}H_0^jH_m)\sum^k_{i=0}x^i
f_i(y)
\in F^{\mathfrak{hv}}.
\end{eqnarray}
Based on Proposition 
\ref{pro4.1} and considering the coefficient of $m^{k}$ in \eqref{55.2}, one has  $xf_k(y)\in F^{\mathfrak{hv}}.$
Therefore, we always get  $$xf_k(y)\in F^{\mathfrak{hv}}$$ for any $\beta(y)\in\mathbb{C}[y]$ with $\beta_0=0$.
Using  the recursive method,   we obtain   $f_0(y),xf_i(y)\in F^{\mathfrak{hv}}$ for $i=1,\ldots,k$.
The claim holds.

Suppose that  $h(y),xg(y)\in F^{\mathfrak{hv}}$ are nonzero polynomials such that $\mathrm{deg}_y(h(y)),\mathrm{deg}_{y}(xg(y))$ are minimal.
Then for any
$f(x,y)=\sum^k_{i=0}
x^if_i(y) \in F^{\mathfrak{hv}}$, it follows from  Claim \ref{claim2} that $h(y)\mid f_0(y)$, $g(y)\mid f_i(y)$  for $i=1,\ldots,k$.

We note that $xh(y) = L_0h(y) \in F^{\mathfrak{hv}}$, we have $g(y)\mid h(y)$. For any $0\neq m\in\mathbb{Z}$, it is clear that  $yg(y)=\frac{1}{m}(\lambda^{-m}H_m-H_0)(xg(y)) \in  F^{\mathfrak{hv}}$. Thus $h(y)
\mid yg(y)$. Therefore, $h(y) = c_1g(y)$  or $h(y) = c_2yg(y)$  for
some nonzero $c_1, c_2 \in\mathbb{C}$. 
 If  $h(y)=c_1g(y)$  for some 
   $c_1\in \mathbb{C}^*$,
one can  see that $F^{\mathfrak{hv}}$ is generated by $g(y)$, namely, $F^{\mathfrak{hv}}= g(y)\mathbb{C}[x,y]$.  
If $h(y)=c_2yg(y)$  for some 
   $c_2\in \mathbb{C}^*$,
we know that $F$ is generated by $xg(y)$ and $yg(y)$, namely, $F^{\mathfrak{hv}}=g(y)(x\mathbb{C}[x,y])+y\mathbb{C}[x,y])$.

{\rm(2)}  We can directly  check it  by (1)
  \end{proof} 
\begin{rema}
The proof  of Claim \ref{claim2} can provide  a   more simple proof than that       Claim of Theorem 3.12    in  \cite{YYX3}, which  mainly depends on Proposition \ref{pro4.1}  and some linear operators.
\end{rema}

By Proposition \ref{le5.1}, any irreducible quotient of the $\mathfrak{hv}$-module $\Psi_{\mathfrak{hv}}(\lambda,\beta)$ is of the form $\Psi_{\mathfrak{hv}}(\lambda,\beta)/R_g^{\mathfrak{hv}}$ for some
$g(y)=y-b$ with $b\in\mathbb{C}^*$ or $\beta(b)\neq0$, which is denoted by $\Phi_{\mathfrak{hv}}(\lambda, \beta, b)$. Then  $\Phi_{\mathfrak{hv}}(\lambda, \beta, b)=\mathbb{C} [x]$. From Proposition \ref{le5.1}   and \eqref{eq2.22},
  the $\mathfrak{hv}$-module structure of $\Phi_{\mathfrak{hv}}(\lambda, \beta, b)$ is given as  
\begin{eqnarray} \label{le5.6}
 L_m f(x) = \lambda^m (x+m\beta(b)) f(x+m),\  
  H_m f(x)= \lambda^m bf(x+m),\ C_if(x)=0
  \end{eqnarray}
for $m\in\Z,i=1,2,3.$
Note that $\mathfrak{hv}$-module $\Phi_{\mathfrak{hv}}(\lambda, \beta, b)$ is irreducible if and only if $\beta(b)\neq0$ or $b\neq0$ (see \cite{CG}).
It follows from  Proposition \ref{le5.1}  that   $S_g^{\mathfrak{hv}}=x\mathbb{C}[x]$
is a submodule if $\beta(b)=b=0$.

\begin{prop}\label{lemm5.2}
    Let $\lambda\in \mathbb{C}^*,b\in\mathbb{C}$ and assume $g(y), g^\prime(y)\in\mathbb{C}[y]$ with
$g(y)=(y-b)g^\prime(y)$.   Then as an   $\mathfrak{hv}$-module $R_{g^\prime}^{\mathfrak{hv}}/R_g^{\mathfrak{hv}} \cong \Phi_{\mathfrak{hv}}(\lambda,\beta,b)$.   If   $(b,\beta(b))\neq(0,0)$, the quotient module is irreducible.
\end{prop}
\begin{proof}
The quotient  space is
$$R_{g^\prime}^\mathfrak{hv}/R_g^\mathfrak{hv} =g^\prime(y)
\mathbb{C}[x].$$
Define the following linear mapping
  \begin{eqnarray*}
\psi: \Phi_{\mathfrak{hv}}(\lambda, \beta, b)&\longrightarrow& R_{g^\prime}^\mathfrak{hv}/R_g^\mathfrak{hv}
\\ f(x)&\longmapsto&g^\prime(y)f(x).
\end{eqnarray*}
According to  \eqref{eq2.22} and \eqref{le5.6}, it is easy to confirm that $\psi$ is an $\mathfrak{hv}$-module isomorphism.
\end{proof}
\begin{coro}
Let $\lambda\in \mathbb{C}^*$ and  let $g(y)$ be a monic polynomial of degree $n$. Assume that  $b_1, \ldots, b_n$ are all roots of $g(y)$ in $\mathbb{C}$ (counting the multiplicity) with $(b_i,\beta(b_i))\neq(0,0)$ for $i=1,\ldots,n$. Set $g_i(y) = (y - b_1)(y - b_2) \cdots (y - b_i)$ and $R_i^\mathfrak{hv} =R_{g_i}^\mathfrak{hv}/R_g^\mathfrak{hv}$ where $i = 1, \ldots, n$. Then
\[
\Psi_\mathfrak{hv}(\lambda, \beta)/R_g^\mathfrak{hv} \supset R_1^\mathfrak{hv} \supset R_2^\mathfrak{hv} \supset \cdots \supset R_{n-1}^\mathfrak{hv} \supset R_n^\mathfrak{hv} = 0
\]
is a decomposition series of the quotient module $\Psi_\mathfrak{hv}(\lambda, \beta)/R_g^\mathfrak{hv}$. Moreover, we obtain the irreducible quotient module
\[
(\Psi_\mathfrak{hv}(\lambda, \beta)/R_g^\mathfrak{hv})/R_1^\mathfrak{hv} \cong \Phi_\mathfrak{hv}(\lambda, \beta, b_1), \quad R_i^\mathfrak{hv}/R_{i+1}^\mathfrak{hv} \cong \Phi_\mathfrak{hv}(\lambda, \beta, b_{i+1}), \quad 1 \leq i \leq n-1.
\]
\end{coro}

\begin{proof}
It can be directly  checked  by Proposition  \ref{lemm5.2} and the following isomorphisms
$$(\Psi_\mathfrak{hv}(\lambda, \beta)/R_g^\mathfrak{hv})/(R_{g_1}^\mathfrak{hv}/R^\mathfrak{hv}_g) \cong \Psi_\mathfrak{hv}(\lambda, \beta)/R_{g_1}^\mathfrak{hv}$$ and 
$$(R_{g_i}^\mathfrak{hv}/R_g^\mathfrak{hv})/(R_{g_{i+1}}^\mathfrak{hv}/R_g^\mathfrak{hv}) \cong R_{g_i}^\mathfrak{hv}/R_{g_{i+1}}^\mathfrak{hv},$$ where  $1 \leq i \leq n-1$.
\end{proof}

\subsection{Modules over the Neveu-Schwarz algebra}
 The Neveu-Schwarz algebra and the Ramond algebra are collectively referred to as the super Virasoro algebra (also called N=1 superconformal algebra), which plays  an important role in mathematical physics. 

From the definition of $\mathfrak{g}$-modules $\Omega(\lambda,\beta)$ in Proposition \ref{pro2.21}, we immediately get the module of  $\mathfrak{ns}$  for $\beta(y)=y$.
\begin{defi}\label{pro5.1}
For $\lambda\in \C^*,f(x,y)\in\C[x,y],f(s,t)\in\C[s,t]$,
 the    $\mathfrak{ns}$-actions
  on $V_{\mathfrak{ns}}=\mathbb{C}[x,y]\oplus\mathbb{C}[s,t]$ are given as 
\begin{eqnarray}
 &&\label{eq555.5}  
 L_m f(x,y) = \lambda^m (x+my) f(x+m,y),  
 \\&&L_m f(s,t)= \lambda^m (s+m(t+\frac{1}{2})) f(s+m,t), 
 \\&& G_p f(x,y)= \lambda^{p-\frac{1}{2}} f(s+p,t),
 \\&&
 G_p f(s,t)= \lambda^{p+\frac{1}{2}} (x + 2py) f(x+p,y),
 \\&&\label{eq555.8}
 C_1f(x,y)=C_1f(s,t)=0
\end{eqnarray}
for $m\in\Z,p\in\mathbb{Z}+\frac{1}{2}.$
 Then   $V_\mathfrak{ns}$ is an   $\mathfrak{ns}$-module under the actions  of \eqref{eq555.5}-\eqref{eq555.8},  and   denoted by $\Psi_{\mathfrak{ns}}(\lambda)$.
\end{defi}
\begin{rema}
 It's not difficult to see that  
 the $\mathfrak{ns}$-module  $\Psi_{\mathfrak{ns}}(\lambda)$ can also be obtained from \cite{YYX3} by the relation  of   N=1  and  N=2 superconformal algebras.
\end{rema}

\begin{prop}\label{lem5.5}
Let $\lambda\in\C^*$.
For  any $g(y)\in\mathbb{C}[y]$, we let $R_g^\mathfrak{ns}=g(y)\mathbb{C}[x,y]\oplus g(t)\mathbb{C}[s,t]\subseteq\Psi_{\mathfrak{ns}}(\lambda) $  and
$S_g^\mathfrak{ns}=g(y)(x\mathbb{C}[x,y]+y\mathbb{C}[x,y])\oplus g(t)\mathbb{C}[s,t] \subseteq\Psi_{\mathfrak{ns}}(\lambda)$. Then the following statements hold.
\begin{itemize}
\item[\rm(1)] The set $\big\{R_g^\mathfrak{ns}, S_g^\mathfrak{ns} \mid g(y)\in \mathbb{C}[y]\big\}$ exhausts all $\mathfrak{ns}$-submodules of $\Psi_{\mathfrak{ns}}(\lambda)$. Moreover, $S_g^\mathfrak{ns} \subseteq R_g^\mathfrak{ns}\subseteq  R_{\hat{g}}^\mathfrak{ns}$  for
any $g(y), \hat{g}(y)\in \mathbb{C} [y]$  with $\hat{g}(y) \mid g(y)$.
\item[\rm(2)] For any nonzero $g(y) \in \mathbb{C}[y]$, the quotient $\Psi_{\mathfrak{ns}}(\lambda)/R_g^\mathfrak{ns}$ is a free $\mathbb{C}[L_0]$-module of rank $2 \mathrm{deg}g(y)$.  Any maximal submodule of $\Psi_{\mathfrak{ns}}(\lambda)$ is of the form $R_g^\mathfrak{ns}$ for some $g(y)\in\mathbb{C}[y]$ with $\mathrm{deg}g(y)=1$.  
\end{itemize}
  \end{prop}
 \begin{proof}
 {\rm (1)} Suppose that   $F^\mathfrak{ns}=
F^\mathfrak{ns}_{\bar0}\oplus F^\mathfrak{ns}_{\bar1}$ is a nonzero submodule of $\Psi_{\mathfrak{ns}}(\lambda)=\mathbb{C}[x,y]\oplus \mathbb{C}[s,t]$. Let $f(x, y)=\sum_{i=0}^kx^if_i(y)$ be a nonzero element in  $\mathbb{C} [x,y]$. By the similar computation in  Claim \ref{claim1}, we know that 
      $f(x,y)=\sum^k_{i=0}x^i
f_i(y)\in F^\mathfrak{ns}_{\bar0}$
if and only if 
   $f_0(y), xf_i(y)\in F^\mathfrak{ns}_{\bar0}$ for $i= 1,\ldots,k$.

Let $h(y),xg(y)\in F^\mathfrak{ns}_{\bar0}$ be nonzero polynomials such that $\mathrm{deg}_y(h(y)),\mathrm{deg}_{y}(xg(y))$ are minimal.
Then for any
$f(x,y)=\sum_k^{i=0}
x_i
f_i(y) \in F^\mathfrak{ns}_{\bar0}$, we have $h(y)\mid f_0(y)$ and $g(y)\mid f_i(y)$  for $i=1,\ldots,k$.

From $xh(y) = L_0h(y) \in F^\mathfrak{ns}_{\bar0}$, we have $g(y)\mid h(y)$. Since 
$$(x+my) (x+m)g(y)= \lambda^{-m}L_m(xg(y)) \in  F^\mathfrak{ns}_{\bar0},$$  by using Proposition \ref{pro4.1}, we have $yg(y)\in  F^\mathfrak{ns}_{\bar0}$. Thus $h(y)
\mid yg(y)$. Therefore, $h(y) = c_1g(y)$  or $h(y) = c_2yg(y)$  for
some nonzero $c_1, c_2 \in\mathbb{C}$. Now we give the following discussion.
\begin{case}
   $h(y)=c_1g(y)$  for some 
   $c_1\in \mathbb{C}^*$. 
\end{case}
In this case, $F^\mathfrak{ns}_{\bar0}$ is generated by $g(y)$, namely, $F^\mathfrak{ns}_{\bar0} = g(y)\mathbb{C}[x,y]$.   For any $u(s,t) \in \mathbb{C}[s,t]$, we have $g(t)u(s,t)=G_{\frac{1}{2}}g(y)u(x-\frac{1}{2},y)\in F^\mathfrak{ns}_{\bar1}$.
Hence, $g(t)\mathbb{C}[s,t]\in F^\mathfrak{ns}_{\bar1}$. On the other hand,  for any $p\neq \frac{1}{2}$ and  $r(s,t)\in F^\mathfrak{ns}_{\bar1}$, it is easy to check that 
\begin{eqnarray*}
    xr(x,y)&=&\frac{2p}{1-2p}\big(\frac{1}{2p}\lambda^{-p-\frac{1}{2}} G_{p}r(s-p,t)-\lambda^{-1} G_{\frac{1}{2}}r(s-\frac{1}{2},t)\big)
    \\&&\in F^\mathfrak{ns}_{\bar0}=g(y)\mathbb{C}[x,y],
    \end{eqnarray*}
which  shows that $r(s,t)\in g(t)\mathbb{C}[s,t],$ namely, $F^\mathfrak{ns}_{\bar1}\subseteq g(t)\mathbb{C}[s,t].$
Thus, $F^\mathfrak{ns}=R_g^\mathfrak{ns}$.
\begin{case}
   $h(y)=c_2yg(y)$  for some 
   $c_2\in \mathbb{C}^*$. 
\end{case}
In this case, $F^\mathfrak{ns}_{\bar0}$ is generated by $yg(y)$ and $xg(y)$, namely, $F^\mathfrak{ns}_{\bar0}=g(y)(x\mathbb{C}[x,y]+y\mathbb{C}[x,y])$. For any $u(s,t)\in \mathbb{C}[s,t]$,  we get
\begin{eqnarray*}
    g(t)u(s,t)=\lambda^{-1}G_{\frac{3}{2}}g(y)xu(x-\frac{3}{2},y)-G_{\frac{1}{2}}g(y)xu(x-\frac{1}{2},y)\in F^\mathfrak{ns}_{\bar1}.
\end{eqnarray*}
Hence, $g(t)\mathbb{C}[s,t]\subseteq F^\mathfrak{ns}_{\bar1}$. On the other hand,   for any $p\neq\frac{1}{2}$ and $r(s,t)\in F^\mathfrak{ns}_{\bar1}$, we have
\begin{eqnarray*}
    xr(x,y)&=&\frac{2p}{1-2p}\big(\frac{1}{2p}\lambda^{-p-\frac{1}{2}} G_{p}r(s-p,t)-\lambda^{-1} G_{\frac{1}{2}}r(s-\frac{1}{2},t)\big)
    \\&&\in F^\mathfrak{ns}_{\bar0}=g(y)(x\mathbb{C}[x,y]+y\mathbb{C}[x,y]),
    \end{eqnarray*}
which gives  $r(s,t)\in g(t)\mathbb{C}[s,t],$ namely, $F^\mathfrak{ns}_{\bar1}\subseteq g(t)\mathbb{C}[s,t].$
Consequently, $F^\mathfrak{ns}_{\bar1}= g(t)
\mathbb{C}[s,t]$ and $F^\mathfrak{ns}=S_g^\mathfrak{ns}.$

{\rm(2)} We can directly  check this result from (1).
  \end{proof}

Based  on  Proposition \ref{lem5.5}, any irreducible quotient of the $\mathfrak{ns}$-module $\Psi_{\mathfrak{ns}}(\lambda)$ is of the form $\Psi_{\mathfrak{ns}}(\lambda)/R_g^\mathfrak{ns}$ for some
$g(y)=y-b$ with $b\in\mathbb{C}^*$, which is denoted by $\Phi_\mathfrak{ns}(\lambda,   b)$. Then as a $\mathbb{Z}_2$-graded vector space, $\Phi_\mathfrak{ns}(\lambda,   b)=\mathbb{C} [x]\oplus\mathbb{C} [s]$ with $\Phi_\mathfrak{ns}(\lambda,   b)_{\bar0} = \mathbb{C} [x]$ and $\Phi_\mathfrak{ns}(\lambda,   b)_{\bar1} = \mathbb{C}[s]$. It follows from Proposition \ref{lem5.5} and \eqref{eq555.5}-\eqref{eq555.8}
that the $\mathfrak{ns}$-module structure of $\Phi_\mathfrak{ns}(\lambda,   b)$ is given by
\begin{eqnarray*}
 &&\label{eq5.12} L_m f(x) = \lambda^m (x+mb) f(x+m),  
 \\&&L_m f(s)= \lambda^m (s+m(b+\frac{1}{2})) f(s+m), 
 \\&& G_p f(x)= \lambda^{p-\frac{1}{2}} f(s+p),
 \\&&
 \label{eq5.15}G_p f(s)= \lambda^{p+\frac{1}{2}} (x + 2pb) f(x+p),
 \\&&C_1f(s)=C_1f(x)=0
\end{eqnarray*}
for $m\in\Z,p\in\mathbb{Z}+\frac{1}{2}.$ It is clear that $\mathfrak{ns}$-module $\Phi_{\mathfrak{ns}}(\lambda,  b)$ is irreducible if and only if  $b\neq0$ (see \cite{YYX1}). From Proposition \ref{lem5.5}, we see that $S_g^\mathfrak{ns}=x\mathbb{C}[x]\oplus\mathbb{C}[s]$
is a submodule if $b=0$.

  \begin{prop}\label{lemma5.666}
    Keep notations as in Proposition \ref{lem5.5}. 
    Let $\lambda\in \mathbb{C}^*,b\in\mathbb{C}$ and assume $g(y), g^\prime(y)\in\mathbb{C}[y]$ with
$g(y)=(y-b)g^\prime(y)$. Then as an   $\mathfrak{ns}$-module $R_{g^\prime}^{\mathfrak{ns}}/R_g^\mathfrak{ns} \cong \Phi_\mathfrak{ns}(\lambda,   b)$.  In particular, if $b\neq0$, the quotient module is irreducible.
\end{prop}

\begin{proof}
    As a $\mathbb{Z}_2$-graded vector space,
$$R_{g^\prime}^{\mathfrak{ns}}/R_g^\mathfrak{ns} = (R_{g^\prime}^{\mathfrak{ns}}/R_g^\mathfrak{ns})_{\bar0}\oplus (R_{g^\prime}^{\mathfrak{ns}}/R_g^\mathfrak{ns})_{\bar1}=g^\prime(y)
\mathbb{C}[x]\oplus g^\prime(t)\mathbb{C}[s].$$
Define the following linear mapping
  \begin{eqnarray*}
\psi: \Phi_{\mathfrak{ns}}(\lambda,b)&\longrightarrow& R_{g^\prime}^{\mathfrak{ns}}/R_g^\mathfrak{ns}
\\ f(x)&\longmapsto&g^\prime(y)f(x)
\\  f(s)&\longmapsto& g^\prime(t)f(s).
\end{eqnarray*}
It is easy to    check that $\psi$ is an $\mathfrak{ns}$-module isomorphism. The proposition clears.
\end{proof}
\begin{coro}
    Keep notations as in Proposition \ref{lem5.5}. Let  $\lambda,b_i\in \mathbb{C}^*$, and  let $g(y)$ be a monic polynomial of degree $n$ for $i=1,\ldots,n$. Suppose that  $b_1, \ldots, b_n$ are all roots of $g(y)$ in $\mathbb{C}^*$ (counting the multiplicity). Set $g_i(y) = (y - b_1)(y - b_2) \cdots (y -b_i)$ and $R_i^\mathfrak{ns} = R_{g_i}^\mathfrak{ns}/R_g^\mathfrak{ns}$ for $i = 1, \ldots, n$. Then
\[
\Psi(\lambda)/R_g^\mathfrak{ns} \supset R_1^\mathfrak{ns} \supset R_2^\mathfrak{ns} \supset \cdots \supset R_{n-1}^\mathfrak{ns} \supset R_n^\mathfrak{ns} = 0
\]
is a decomposition series of the quotient module $\Psi(\lambda)/R_g^\mathfrak{ns}$. Moreover, we have the irreducible quotient modules
\[
(\Psi(\lambda)/R_g^\mathfrak{ns})/R_1^{\mathfrak{ns}} \cong \Phi_{\mathfrak{ns}}(\lambda,  b_1), \quad R_i^{\mathfrak{ns}}/R_{i+1}^{\mathfrak{ns}} \cong \Phi_{\mathfrak{ns}}(\lambda,  b_{i+1}), \quad 1 \leq i \leq n-1.
\]
\end{coro}

\begin{proof}
It follows   from Proposition \ref{lem5.5} and the following isomorphisms
$(\Psi_{\mathfrak{ns}}(\lambda)/R_g^\mathfrak{ns})/(R_{g_1}^{\mathfrak{ns}}/R_g^\mathfrak{ns}) \cong \Psi_{\mathfrak{ns}}(\lambda)/R_{g_1}^{\mathfrak{ns}}$ and
$(R_{g_i}^{\mathfrak{ns}}/R_g^\mathfrak{ns})/(R_{g_{i+1}}^{\mathfrak{ns}}/R_g^\mathfrak{ns}) \cong R_{g_i}^{\mathfrak{ns}}/R_{g_{i+1}}^{\mathfrak{ns}}$   for  $1 \leq i \leq n-1.$
\end{proof}

\subsection{Modules over  the Fermion-Virasoro algebra}
We first recall a class of non-weight modules defined  in \cite{XZ}, which are closely associated with the classification of $U(\mathbb{C}L_0)$-free modules of rank 2 over $\mathfrak{fv}$.
Assume that $\lambda\in\mathbb{C}^*,\beta(t)\in\mathbb{C}[t],b\in\mathbb{C}$. Let $\Phi_{\mathfrak{fv}}(\lambda,
\beta,b)$
be a polynomial algebra with the indeterminate elements  $x,s$. For any $f (x) \in \mathbb{C}[x],f (s) \in \mathbb{C}[s]$, we define the following nontrivial $\mathfrak{fv}$-actions   on  $\Phi_{\mathfrak{fv}}(\lambda,
\beta,b)$ as follows:
\begin{eqnarray*}
 &&  
 L_m f(x) = \lambda^m (x+m\beta(b)) f(x+m), 
 \\&&L_m f(s)= \lambda^m (s+m(\beta(b)+\frac{1}{2})) f(s+m), 
 \\&&
 Q_p f(s)= \lambda^{p+\frac{1}{2}} bf(x+p)
\end{eqnarray*}
for $m\in\Z,p\in\mathbb{Z}+\frac{1}{2}.$
From \cite{XZ}, we know that $\Phi_{\mathfrak{fv}}(\lambda,
\beta,b)$ is a reducible module.

Note that $\mathfrak{fv}$ is a subalgebra of $\mathfrak{g}$.
By the definition of $\mathfrak{g}$-modules $\Omega(\lambda,\beta)$ in Proposition \ref{pro2.21}, we  get a class of   modules over $\mathfrak{fv}$ as follows.
\begin{defi}
\label{pro5.8}
For $\lambda\in \C^*,f(x,y)\in\C[x,y],f(s,t)\in\C[s,t]$,
 the   $\mathfrak{fv}$-actions
  on $V_{\mathfrak{fv}}=\mathbb{C}[x,y]\oplus\mathbb{C}[s,t]$ are given by
\begin{eqnarray}
 &&  \label{5.161616}
 L_m f(x,y) = \lambda^m (x+m\beta(y)) f(x+m,y),  
 \\&&L_m f(s,t)= \lambda^m (s+m(\beta(t)+\frac{1}{2})) f(s+m,t), 
  \\&& Q_p f(x,y)=C_if(x,y)=C_if(s,t)=0,
 \\&&\label{5.19}
 Q_p f(s,t)= \lambda^{p+\frac{1}{2}} yf(x+p,y)
\end{eqnarray}
for $m\in\Z,p\in\mathbb{Z}+\frac{1}{2},i=1,3.$
 Then we see that  $V_{\mathfrak{fv}}$ is an    $\mathfrak{fv}$-module under the actions  of \eqref{5.161616}-\eqref{5.19},  and   denoted by $\Psi_{\mathfrak{fv}}(\lambda,
\beta)$.
\end{defi}

\begin{prop}\label{lemm5.9999}
\begin{itemize}
\item[\rm(1)] For $\alpha\in\mathbb{C}$,  then
 $\Psi_{\mathfrak{fv}}(\lambda,
\beta)$   has the following submodule filtration 
\begin{eqnarray*}  
&&\cdots\subseteq (y-\alpha)^n\mathbb{C}[x,y] \oplus (t-\alpha)^n\mathbb{C}[s,t] 
\subseteq(y-\alpha)^{n-1}\mathbb{C}[x,y] 
\oplus (t-\alpha)^{n-1}\mathbb{C}[s,t]
\\&&
\subseteq \cdots \subseteq\mathbb{C}[x,y] \oplus  \mathbb{C}[s,t];
\end{eqnarray*}
\item[\rm(2)] 
Let $b\in\mathbb{C}$ and   $g(y), g^\prime(y)\in\mathbb{C}[y]$ with
$g(y)=(y-b)g^\prime(y)$,  and let $M_g^{\mathfrak{fv}}=g(y)\mathbb{C}[x,y]\oplus g(t)\mathbb{C}[s,t]\subseteq\Psi_{\mathfrak{hv}}(\lambda,\beta) $. Then as an  $\mathfrak{fv}$-module $M_{g^\prime}^{\mathfrak{fv}}/M_g^{\mathfrak{fv}} \cong \Phi_{\mathfrak{fv}}(\lambda,\beta,b)$.
\end{itemize}
\end{prop}
\begin{proof}
\begin{itemize}
\item[\rm(1)]
    We only need to show that 
$(y-\alpha)^i\mathbb{C}[x,y] \oplus (t-\alpha)^i\mathbb{C}[s,t]$ for $i=1,2,\ldots,n$ are submodules. But this can be checked directly by the definition of $\mathfrak{fv}$ and Definition \ref{pro5.8}.
\item[\rm(2)]
As a $\mathbb{Z}_2$-graded vector space,
$$M_{g^\prime}^{\mathfrak{fv}}/M_g^{\mathfrak{fv}} = (M_{g^\prime}^{\mathfrak{fv}}/M_g^{\mathfrak{fv}})_{\bar0}\oplus (M_{g^\prime}^{\mathfrak{fv}}/M_g^{\mathfrak{fv}})_{\bar1}=g^\prime(y)
\mathbb{C}[x]\oplus g^\prime(t)\mathbb{C}[s].$$
We define the following linear mapping
  \begin{eqnarray*}
\psi: \Phi_{\mathfrak{fv}}(\lambda,\beta,b)&\longrightarrow& M_{g^\prime}^{\mathfrak{fv}}/M_g^{\mathfrak{fv}}
\\ f(x)&\longmapsto&g^\prime(y)f(x)
\\  f(s)&\longmapsto& g^\prime(t)f(s).
\end{eqnarray*}
It is easy to show that $\psi$ is an $\mathfrak{fv}$-module isomorphism.
\end{itemize}
\end{proof}

\subsection{Modules over  the
Heisenberg-Clifford superalgebra}
 The
Heisenberg-Clifford superalgebra $\mathfrak{hc}$ is a subalgebra of $\mathfrak{g}$.
Let $\lambda\in\mathbb{C}^*$, $b\in\mathbb{C}$,  and  let $\Phi_{\mathfrak{hc}}(\lambda,
b)$
be a polynomial algebra with indeterminate elements  $x,s$. For any $f (x) \in \mathbb{C}[x]$ and $f (s) \in \mathbb{C}[s]$, we define nontivial $\mathfrak{hc}$-actions   on  $\Phi_{\mathfrak{hc}}(\lambda,
b)$ as follows:
\begin{eqnarray*}
 &&
 H_m f(x)= \lambda^m bf(x+m),
 \    H_m f(s)= \lambda^m bf(s+m),
 \   Q_p f(s)= \lambda^{p+\frac{1}{2}} bf(x+p)
\end{eqnarray*}
for $m\in\Z,p\in\mathbb{Z}+\frac{1}{2}.$
Clearly, $\Phi_{\mathfrak{hc}}(\lambda,
b)$ is a reducible $\mathfrak{hc}$-module. 
From the definition of $\mathfrak{g}$-modules $\Omega(\lambda,\beta)$ in Proposition \ref{pro2.21}, we immediately obtain  a class of  modules over  $\mathfrak{hc}$ as follows.  
\begin{defi}
    \label{pro5.2221}
For $\lambda\in \C^*,f(x,y)\in\C[x,y],f(s,t)\in\C[s,t]$,
 the   $\mathfrak{hc}$-actions
  on $V_{\mathfrak{hc}}=\mathbb{C}[x,y]\oplus\mathbb{C}[s,t]$ are given by
\begin{eqnarray}
 &&\label{5.20}  
 H_m f(x,y)= \lambda^m yf(x+m,y),
 \\&& H_m f(s,t)= \lambda^m tf(s+m,t),
  \\&&  Q_p f(x,y)=C_3 f(x,y)=C_3f(s,t)=0,
 \\&&\label{5.23}  Q_p f(s,t)= \lambda^{p+\frac{1}{2}} yf(x+p,y)
\end{eqnarray}
for $m\in\Z,p\in\mathbb{Z}+\frac{1}{2}.$
 Then we see that  $V_\mathfrak{hc}$ is an    $\mathfrak{hc}$-module under the actions  of \eqref{5.20}-\eqref{5.23},  and   denoted by $\Psi_{\mathfrak{hc}}(\lambda)$.
\end{defi}

 \begin{prop}\label{lemma51111}
\begin{itemize}
\item[\rm(1)] For $\alpha\in\mathbb{C}$,  then
 $\Psi_{\mathfrak{hc}}(\lambda)$   has the following submodule filtration 
\begin{eqnarray*}  
&&\cdots\subseteq (y-\alpha)^n\mathbb{C}[x,y] \oplus (t-\alpha)^n\mathbb{C}[s,t] 
\subseteq(y-\alpha)^{n-1}\mathbb{C}[x,y] 
\oplus (t-\alpha)^{n-1}\mathbb{C}[s,t]
\\&&\subseteq \cdots \subseteq\mathbb{C}[x,y] \oplus  \mathbb{C}[s,t];
\end{eqnarray*}
\item[\rm(2)] 
Let $b\in\mathbb{C}$ and   $g(y), g^\prime(y)\in\mathbb{C}[y]$ with
$g(y)=(y-b)g^\prime(y)$,  and let $M_g^{\mathfrak{hc}}=g(y)\mathbb{C}[x,y]\oplus g(t)\mathbb{C}[s,t]\subseteq\Psi_{\mathfrak{hc}}(\lambda) $. Then as an  $\mathfrak{hc}$-module $M_{g^\prime}^{\mathfrak{hc}}/M_g^{\mathfrak{hc}} \cong \Phi_{\mathfrak{hc}}(\lambda,b)$.
\end{itemize}
\end{prop}
\begin{proof}
By the similar description in the proof of Proposition \ref{lemm5.9999}, we  obtain (1) and (2).
\end{proof}

 \section{Intermediate series modules}
In this section, we shall realize the intermediate series $\mathfrak{g}$-modules by using  the weighting functor $\mathcal{W}$ introduced in \cite{N2} and $\Omega(\lambda,\beta)$.
From Proposition \ref{pro2.21}, we  rewrite the definition of $\mathfrak{g}$-module   $\Omega(\lambda,\beta)$ as follows. 
\begin{defi}
For $\lambda\in \C^*,\beta(H_0)\in\C[H_0],f(L_0,H_0)\in\C[L_0,H_0]$,
 the nontrivial  $\mathfrak{g}$-actions
  on $V=\mathbb{C}[L_0,H_0]\mathbf{1}_{\bar0}\oplus\mathbb{C}[L_0,H_0]\mathbf{1}_{\bar1}$ are defined by
\begin{eqnarray*}
 &&\label{2.37788}  
 L_m f(L_0,H_0)\mathbf{1}_{\bar0} = \lambda^m (L_0+m\beta(H_0)) f(L_0+m,H_0)\mathbf{1}_{\bar0},  
 \\&&\label{2.677788}L_m f(L_0,H_0)\mathbf{1}_{\bar1}= \lambda^m (L_0+m(\beta(H_0)+\frac{1}{2})) f(L_0+m,H_0)\mathbf{1}_{\bar1}, 
 \\&& \label{2.477788} H_m f(L_0,H_0)\mathbf{1}_{\bar0}= \lambda^m H_0f(L_0+m,H_0)\mathbf{1}_{\bar0},
 \\&&\label{2.777788}H_m f(L_0,H_0)\mathbf{1}_{\bar1}= \lambda^m H_0f(L_0+m,H_0)\mathbf{1}_{\bar1},
 \\&& \label{2.577788}G_p f(L_0,H_0)\mathbf{1}_{\bar0}= \lambda^{p-\frac{1}{2}} f(L_0+p,H_0)\mathbf{1}_{\bar1},
 \\&& \label{2.877888}G_p f(L_0,H_0)\mathbf{1}_{\bar1}= \lambda^{p+\frac{1}{2}} (L_0 + 2p\beta(H_0)) f(L_0+p,H_0)\mathbf{1}_{\bar0},
 \\&&\label{2.977788}  Q_p f(L_0,H_0)\mathbf{1}_{\bar1}= \lambda^{p+\frac{1}{2}} H_0f(L_0+p,H_0)\mathbf{1}_{\bar0}
 \end{eqnarray*}
for $m\in\Z,p\in\mathbb{Z}+\frac{1}{2}.$
 Then $V$ is a   $\mathfrak{g}$-module under the above actions,   and is  denoted by $\Omega^\prime(\lambda,
\beta)$.
\end{defi}
\begin{prop}[\cite{LL}]
   \label{intermediate} 
Let  $a,b,c\in\C$. { The intermidate series modules} $A(a,b,c)$ over $\R$   is a module
with a basis $\big\{v_n^+,v_q^{-}\mid n\in\Z,q\in\mathbb{Z}+\frac{1}{2}\big\}$ satisfying the nontrivial actions as follows
\begin{eqnarray*}
&&L_mv_n^+ =(c-n+ma)v_{n+m}^+,\\
&&L_mv_q^{-}=(c-q+m(a+\frac{1}{2}))v_{q+m}^{-},\\&&
 H_mv_n^+=b v_{n+m}^+,
 \ H_mv_q^{-}=b v_{q+m}^{-},   
  \\&& G_pv_q^{-}=(c-q+2p(a+\frac{1}{2}) )v_{p+q}^+,\\&&
 G_pv_n^+=v_{p+n}^-, \ Q_pv_q^-=b v_{p+q}^+, 
\end{eqnarray*}
where $m\in\Z,p\in\mathbb{Z}+\frac{1}{2}$.   Furthermore,  
  ${A(a,b,c)}$   is reducible if and only if  $c\in\mathbb{Z},a=-1,b=0$ or $c\in\mathbb{Z}+\frac{1}{2},a=-\frac{1}{2},b=0$.
  \end{prop}

 For $n\in\mathbb{Z}$, $\epsilon=\{0,\frac{1}{2}\}$, $\alpha=(\alpha_1,\alpha_2)\in\mathbb{C}\times\mathbb{C}$, let 
  $\mathcal{I}_{n+\alpha+\epsilon}$
    generated by $(L_0+n+\epsilon+\alpha_1)$ and $(H_0+\alpha_2)$
   be the maximal ideal of $U(\mathfrak{h})=\mathbb{C}[L_0,H_0]$.
 For a $\mathfrak{g}$-module $M$, set  $M_{n+\epsilon+\alpha}=M/\mathcal{I}_{n+\epsilon+\alpha} M$. 
Denote $$\mathcal{W}(M)=\bigoplus_{r\in\mathbb{Z}+\epsilon}M_{r+\alpha}.$$
It follows from Proposition $8$  in \cite{N2} that we  check the following
construction.
\begin{prop} Let $\epsilon=\{0,\frac{1}{2}\}$  and $\alpha=(\alpha_1,\alpha_2)\in\mathbb{C}\times\mathbb{C}$.
 The vector space $\mathcal{W}(M)$ becomes a $\mathfrak{g}$-module under the following nontrivial actions:
\begin{eqnarray*}
   && L_m\cdot(v+\mathcal{I}_{n+\epsilon+\alpha}M)= L_mv+\mathcal{I}_{m+n+\epsilon+\alpha}M,
   \\&&
 H_m\cdot(v+\mathcal{I}_{n+\epsilon+\alpha}M)= H_mv+\mathcal{I}_{m+n+\epsilon+\alpha}M,
 \\&&
G_p\cdot(v+\mathcal{I}_{n+\epsilon+\alpha}M)= G_pv+\mathcal{I}_{p+n+\frac{1}{2}+\alpha}M,
\\&&Q_p\cdot(v+\mathcal{I}_{n+\epsilon+\alpha}M)= Q_pv+\mathcal{I}_{p+n+\frac{1}{2}+\alpha}M,
\end{eqnarray*}
where $m\in\mathbb{Z},p,q\in\mathbb{Z}+\frac{1}{2}$.   
\end{prop}

Now we present  the main result of this section.
\begin{theo}\label{th664}
Let $\lambda\in\C^*$, $\alpha=(\alpha_1,\alpha_2)\in\mathbb{C}\times\mathbb{C}$ and $\beta(H_0)\in\mathbb{C}[H_0]$. As $\R$-modules, we obtain $\mathcal{W}(\Omega^\prime(\lambda,\beta))\cong A\big(\beta(-\alpha_2)-1,-\alpha_2,-\alpha_1\big)$.
\end{theo}
\begin{proof}
Let $\mathbf{1}_{\bar0}$ and $\mathbf{1}_{\bar1}$  be two generators   of $\Omega^\prime(\lambda,\beta)$.
For $n\in\mathbb{Z},q\in\mathbb{Z}+\frac{1}{2},\alpha=(\alpha_1,\alpha_2)\in\mathbb{C}\times\mathbb{C}$, setting  \begin{eqnarray*}
&&v_n^+=\mathbf{1}_{\bar0}+\mathcal{I}_{n+\alpha}
\Omega^\prime(\lambda,\beta)\in\Omega^\prime(\lambda,\beta)_{\bar0},
\\&&v_q^{-}=\mathbf{1}_{\bar1}+\mathcal{I}_{q+\alpha}
\Omega^\prime(\lambda,\beta)\in
\Omega^\prime(\lambda,\beta)_{\bar1},
\end{eqnarray*}
then we have the non-trivial   actions:
\begin{eqnarray*}
\nonumber L_mv_{n}^+&=&\lambda^m(L_0+m\beta(H_0))\mathbf{1}_{\bar0}+\mathcal{I}_{m+n+\alpha}
\Omega^\prime(\lambda,\beta)
\\&=&(-\alpha_1-n+m(\beta(-\alpha_2)-1))\lambda^mv_{n+m}^+,
\\ \nonumber L_mv_{q}^{-}&=&\lambda^m(L_0+m(\beta(H_0)+\frac{1}{2}))\mathbf{1}_{\bar1}+\mathcal{I}_{m+q+\alpha}
\Omega^\prime(\lambda,\beta)
\\&=&(-\alpha_1-q+m(\beta(-\alpha_2)-\frac{1}{2}))\lambda^mv_{q+m}^{-},
\\
H_mv_{n}^+&=&\lambda^mH_0\mathbf{1}_{\bar0}+\mathcal{I}_{m+n+\alpha}
\Omega^\prime(\lambda,\beta)
=-\alpha_2\lambda^mv_{n+m}^+,
\\H_mv_{q}^{-}&=&\lambda^mH_0\mathbf{1}_{\bar1}+\mathcal{I}_{m+q+\alpha}
\Omega^\prime(\lambda,\beta)
=-\alpha_2\lambda^mv_{q+m}^{-},
\\ 
G_pv_n^+&=&\lambda^{p-\frac{1}{2}}\mathbf{1}_{\bar0}+\mathcal{I}_{n+p+\alpha}
\Omega^\prime(\lambda,\beta)=\lambda^{p-\frac{1}{2}}v_{n+p}^{-},
\\ G_pv_{q}^{-}&=&\lambda^{p+\frac{1}{2}}(L_0+2p\beta(H_0))\mathbf{1}_{\bar1}+\mathcal{I}_{p+q+\alpha}
\Omega^\prime(\lambda,\beta)
\\&=&\label{12qwa418} (-\alpha_1-q+p(2\beta(-\alpha_2)-1))\lambda^{p+\frac{1}{2}}v_{p+q}^{+},
\\ Q_pv_{q}^{-}&=&\lambda^{p+\frac{1}{2}}H_0\mathbf{1}_{\bar1}+\mathcal{I}_{p+q+\alpha}
\Omega^\prime(\lambda,\beta)=-\alpha_2\lambda^{p+\frac{1}{2}}v_{p+q}^+
\end{eqnarray*}
for $m\in\Z,p\in\mathbb{Z}+\frac{1}{2}$.
Now   letting $\hat{v}_n^+=\lambda^nv_{n}^+,
\hat{v}_q^-=\lambda^{q-\frac{1}{2}}v_{q}^-$ in the above equations,  one can see that
\begin{eqnarray*}
&&L_m\hat{v}_{n}^+=(-\alpha_1-n+m(\beta(-\alpha_2)-1))\hat{v}_{n+m}^+,
\\&&
 L_m\hat{v}_{q}^{-}=(-\alpha_1-q+m(\beta(-\alpha_2)-\frac{1}{2}))\hat{v}_{q+m}^{-},
\\&&H_m\hat{v}_{n}^+=-\alpha_2\hat{v}_{n+m}^+,\
H_m\hat{v}_{q}^{-}=-\alpha_2\hat{v}_{q+m}^{-},
 \\&& G_p\hat{v}_{q}^-=(-\alpha_1-q+p(2\beta(-\alpha_2)-1))\hat{v}_{p+q}^{+},
 \\&&
 G_p\hat{v}_{n}^{+}= \hat{v}_{n+p}^{-}, \ Q_p\hat{v}_{q}^{-}= -\alpha_2\hat{v}_{n+p}^{+}.
\end{eqnarray*}
Therefore, we confirm that   $\mathcal{W}(\Omega^\prime(\lambda,\beta))\cong A\big(\beta(-\alpha_2)-1,-\alpha_2,-\alpha_1\big)$.
 
\end{proof}
\begin{rema}
  By the similar method, we also can get  the  weight $\mathfrak{g}$-modules $A(\beta(b)-1,-b,-\alpha)$ from non-weight modules  $\Phi(\lambda,\beta,b)$ by letting $(L_0+n+\epsilon+\alpha)$  
   be the maximal ideal of $\mathbb{C}[L_0]$.  Therefore,  using the weighting  functor,  the same weight $\mathfrak{g}$-module  may be obtained from different non-weight $\mathfrak{g}$-modules. In other words, during applications of   the   weighting  functor, it  may  reduce the types of modules.
\end{rema}

\section*{Authors' contributions}
All authors contributed equally to this work.

\section*{Data availability}
This manuscript has no associated data.

\section*{Conflict of interest }
The authors declare that they have no conflict of interest.

\section*{Acknowledgements}
This work was supported by the National Natural Science Foundation of China (Nos. 12171129, 11971350),    Fujian Alliance of Mathematics (No. 2023SXLMMS05).
Finally‌, we would like to thank ‌the referee‌ for their valuable comments and suggestions.

\bigskip

Ziqi Hong

\vspace{2pt}
 School of Science, Jimei University, Xiamen, Fujian 361021, China

\vspace{2pt}
 ziqihong527@163.com

\bigskip

Haibo Chen

\vspace{2pt}
 School of Science, Jimei University, Xiamen, Fujian 361021, China

\vspace{2pt}
hypo1025@jmu.edu.cn

\bigskip

Yucai Su

\vspace{2pt}
School of Science, Jimei University, Xiamen, Fujian 361021, China

\vspace{2pt}
yucaisu@jmu.edu.cn

\end{document}